\documentclass{amsart}
\usepackage{amsmath}
\usepackage{amscd}
\usepackage{amssymb}
\usepackage[mathscr]{eucal}
\usepackage{multirow}
\usepackage{bm}
\usepackage[pdftex]{graphicx}
\input xy
\xyoption{all}
\newcommand{\s}{\mathop{\mathrm{S}}\nolimits}
\newtheorem{thm}{Theorem}[section]
\newtheorem{prop}[thm]{Proposition}
\newtheorem{lemma}[thm]{Lemma}
\newtheorem{cor}[thm]{Corollary}
\newtheorem{conj}[thm]{Conjecture}
\newtheorem{thm'}{Theorem}[subsection]
\newtheorem{prop'}[thm']{Proposition}
\newtheorem{lemma'}[thm']{Lemma}
\newtheorem{cor'}[thm']{Corollary}

\theoremstyle{definition}

\newtheorem{defi'}[thm']{Definition}

\theoremstyle{remark}
\newtheorem{prob}[thm]{Problem}
\newtheorem{rem}[thm]{Remark}
\newtheorem{rem'}[thm']{Remark}

\newtheorem{exam'}[thm']{Example}
\newtheorem{prob'}[thm']{Problem}

\newtheorem{quest'}[thm']{Question}
\numberwithin{equation}{section}
\author{Hideaki \=Oshima}
\address[H. \=Oshima]{Professor Emeritus, Ibaraki University, Mito, Ibaraki 310-8512, Japan}
\email{hideaki.ooshima.mito@vc.ibaraki.ac.jp}

\author{Katsumi \=Oshima}
\address[K. \=Oshima]{1-4001-5 Ishikawa, Mito, Ibaraki 310-0905, Japan}
\email{k-oshima@mbr.nifty.com}

\begin{document}
\allowdisplaybreaks
\title{A system of unstable higher Toda brackets}
\maketitle

\begin{abstract}
We show that a system of unstable higher Toda brackets can be defined inductively. 
\end{abstract}
 
\section{Introduction}
The Toda bracket is one of the basic tools in homotopy theory. 
After \cite{T1,T2} a number of definitions of higher Toda brackets have appeared in the literature. 
Nowadays theories of higher Toda brackets have become quite 
abstract and categorical (see for example \cite{BBG,BBS,BBS2,BJT} and references there), and 
few of them deal in {\it subscripted} brackets. 
We would like to study a topological and not so abstract theory of subscripted brackets.  
Even for topological higher Toda brackets, 
it seems difficult to nominate one of known theories as the standard one. 
In \cite{OO2,OO3}, we defined two systems of unstable higher Toda brackets 
as candidates for the standard system. 
Our definitions were basically just postulated and so one may feel that they are not 
really defined. 
The purpose of the present paper is to clear a little such a doubt 
by proving that one of two systems can be defined inductively. 

We will use notations of \cite{OO2,OO3} freely, while some of them shall be listed in Section~2. 

Given the data
\begin{equation}
\left\{\begin{array}{@{\hspace{0.2mm}}ll}
n,\ &\text{an integer $\ge 3$},\\
(X_{n+1},\dots,X_1),\ &\text{a sequence of well-pointed spaces},\\
\vec{\bm m}=(m_n,\dots,m_1),\ &\text{a sequence of non-negative integers},\\
\vec{\bm f}=(f_n,\dots,f_1),\ &\text{a sequence of pointed maps $f_k:\Sigma^{m_k}X_k\to X_{k+1}$},
\end{array}\right.
\end{equation}
we defined the unstable $n$-fold Toda bracket $\{\vec{\bm f}\,\}^{(\ddot{s}_t)}_{\vec{\bm m}}\subset 
[\Sigma^{n-2}\Sigma^{m_n}\cdots\Sigma^{m_1}X_1,X_{n+1}]$ in \cite{OO2,OO3} 
(see Section 2 below). 
By \cite[Theorem 7.1]{OO3}, $\{f_3,f_2,f_1\}^{(\ddot{s}_t)}_{(m_3,m_2,m_1)}$ is equal to the classical Toda bracket 
$\{f_3,f_2,\Sigma^{m_2}f_1\}_{m_3}$ which was denoted by $\{f_3,\Sigma^{m_3}f_2,\Sigma^{m_3}\Sigma^{m_2}f_1\}_{m_3}$ in \cite{T2}. 
Hence we have  
$$
\{f_3,f_2,f_1\}^{(\ddot{s}_t)}_{(m_3,m_2,m_1)}=\bigcup_{A_2,A_1}\{[f_3,A_2,\Sigma^{m_3}f_2],(\Sigma^{m_3}f_2,\widetilde{\Sigma}^{m_3}A_1,\Sigma^{m_3}\Sigma^{m_2}f_1)\}_{(0,0)},
$$
where the union $\bigcup_{A_2,A_1}$ is taken over all pairs $(A_2,A_1)$ of null homotopies 
$A_2:f_3\circ\Sigma^{m_3}f_2\simeq *$ and $A_1:f_2\circ\Sigma^{m_2}f_1\simeq *$, and $\{g,f\}_{(0,0)}$ is the one point set 
consisting of the homotopy class of $g\circ f$ for any pointed maps $Z\overset{g}{\leftarrow}Y\overset{f}{\leftarrow}X$.  
Moreover when there is a null homotopy $A:g\circ f\simeq *$, that is $A:CX=X\wedge I\to Z$ with $A(x\wedge 0)=g\circ f(x)$, 
the extension of $g$ along $f$, $[g,A,f]:Y\cup_f CX\to Z$, and the coextension of $f$ along $g$, $(g,A,f):\Sigma X\to Z\cup_g CY$, were defined in \cite{Og,OO2} 
(see Section 2 below and \cite{T2} for the original definitions), and, given a non negative integer $m$, 
the map $\widetilde{\Sigma}^mA:C\Sigma^mX\to \Sigma^mZ$ defined by $x\wedge s\wedge t\mapsto A(x\wedge t)\wedge s\ 
(x\in X,\ s\in\s^m,\ t\in I)$ is a null homotopy of $\Sigma^m g\circ\Sigma^m f$ (see Section 2).  

Our main result is

\begin{thm}
For $n\ge 4$, we have 
\begin{align*}
&\{f_n,\dots,f_1\}^{(\ddot{s}_t)}_{(m_n,\dots,m_1)}\\
&=\bigcup_{A_2,A_1}\{f_n,\dots,f_4,[f_3,A_2,\Sigma^{m_3}f_2],(\Sigma^{m_3}f_2,\widetilde{\Sigma}^{m_3}A_1,\Sigma^{m_3}\Sigma^{m_2}f_1)\}^{(\ddot{s}_t)}_{(m_n,\dots,m_4,0,0)}\\
&\hspace{3cm}\circ(1_{\Sigma^{m_3}\Sigma^{m_2}\Sigma^{m_1}X_1}\wedge\tau(\s^{m_4}\wedge\cdots\wedge\s^{m_n},\s^1)\wedge 1_{(\s^1)^{\wedge(n-3)}})\\
&=\bigcup_{\text{all }\vec{\bm A}}\{f_n,\dots,f_4,[f_3,A_2,\Sigma^{m_3}f_2],(\Sigma^{m_3}f_2,\widetilde{\Sigma}^{m_3}A_1,\Sigma^{m_3}\Sigma^{m_2}f_1)\}^{(\ddot{s}_t)}_{(m_n,\dots,m_4,0,0)}\\
&\hspace{3cm}\circ(1_{\Sigma^{m_3}\Sigma^{m_2}\Sigma^{m_1}X_1}\wedge\tau(\s^{m_4}\wedge\cdots\wedge\s^{m_n},\s^1)\wedge 1_{(\s^1)^{\wedge(n-3)}})\\
&=\bigcup_{\text{admissible }\vec{\bm A}}\{f_n,\dots,f_4,[f_3,A_2,\Sigma^{m_3}f_2],(\Sigma^{m_3}f_2,\widetilde{\Sigma}^{m_3}A_1,\Sigma^{m_3}\Sigma^{m_2}f_1)\}^{(\ddot{s}_t)}_{(m_n,\dots,m_4,0,0)}\\
&\hspace{3cm}\circ(1_{\Sigma^{m_3}\Sigma^{m_2}\Sigma^{m_1}X_1}\wedge\tau(\s^{m_4}\wedge\cdots\wedge\s^{m_n},\s^1)\wedge 1_{(\s^1)^{\wedge(n-3)}})
\end{align*}
where the union $\bigcup_{A_2,A_1}$ is taken over all pairs $(A_2,A_1)$ of null homotopies 
$A_2:f_3\circ\Sigma^{m_3}f_2\simeq *$ and $A_1:f_2\circ\Sigma^{m_2}f_1\simeq *$, 
the union $\bigcup_{\text{all }\vec{\bm A}}$ is taken over all sequences $\vec{\bm A}=(A_{n-1},\dots,A_1)$ of null homotopies $A_k:f_{k+1}\circ\Sigma^{m_{k+1}}f_k\simeq *$, 
and the union $\bigcup_{\text{admissible }\vec{\bm A}}$ is taken over all admissible sequences $\vec{\bm A}=(A_{n-1},\dots,A_1)$ of null homotopies $A_k:f_{k+1}\circ\Sigma^{m_{k+1}}f_k\simeq *$. 
Here ``admissible'' means that 
$$
[f_{k+2},A_{k+1},\Sigma^{m_{k+2}}f_{k+1}]\circ(\Sigma^{m_{k+2}}f_{k+1},\widetilde{\Sigma}^{m_{k+2}}A_k,\Sigma^{m_{k+2}}\Sigma^{m_{k+1}}f_k)\simeq *\ (1\le k\le n-2). 
$$
\end{thm}

In the above descriptions, if $A_2$, $A_1$, or $\vec{\bm A}$ does not exist, 
then $\{\vec{\bm f}\,\}^{(\ddot{s}_t)}_{\vec{\bm m}}$ denotes the empty set. 
The above theorem expands \cite[Theorem 6.7.1]{OO2} and \cite[Theorem 8.1]{OO3} and it says that 
$\{\vec{\bm f}\,\}^{(\ddot{s}_t)}_{\vec{\bm m}}$ is the union of classical Toda brackets and that 
we can define a kind of brackets inductively even in the category of pointed spaces $\mathrm{TOP}^*$. 

We have two corollaries below which suggest that higher subscripted brackets are useful in computing homotopy groups of spheres. 

\begin{cor} If $n\ge 3$, $X_{n+1}=\s^{m+1}$, $X_k\ (n\ge k\ge 1)$ is a connected CW-complex with a vertex as the base point, and $m_n\ge 1$, then 
$$
H\{\vec{\bm f}\,\}^{(\ddot{s}_t)}_{\vec{\bm m}}\subset \{H(f_n),f_{n-1},\dots,f_1\}^{(\ddot{s}_t)}_{\vec{\bm m}},
$$
where $H:[\Sigma K,\s^{m+1}]\to[\Sigma K,\s^{2m+1}]$ is the generalized Hopf invariant \cite{J,T2} ($K$ is any connected CW-complex with a vertex as the base point).
\end{cor}
A result similar to the above corollary can be seen in \cite[Proposition 5]{M}. 
By \cite[Example 5.9]{OO1} and \cite{T2} we easily have 

\begin{cor}
Let $\mu_3$ be any element of $\{\eta_3,\nu',8\iota_5,\nu_5\}^{(\ddot{s}_t)}_{(1,1,0,0)}$. 
Then 
$$
\{\eta_3,\nu',8\iota_5,\nu_5\}^{(\ddot{s}_t)}_{(1,1,0,0)}=\mu_3+\mathbb{Z}_2\{\eta_3\circ\varepsilon_4\}\subset\pi_{12}(\s^3)=\mathbb{Z}_2\{\mu_3\}\oplus\mathbb{Z}_2\{\eta_3\circ\varepsilon_4\}.
$$
\end{cor}
In an other paper we will study more applications of subscripted brackets and fill (or correct) gaps (or errors) in circulating papers.

In Section 2, we review the definition of 
$\{\vec{\bm f}\,\}^{(\ddot{s}_t)}_{\vec{\bm m}}$. 
In Section 3, we prove Theorem 1.1. 
In Section 4, we prove Corollary 1.2. 
In Section 5, we give an inductive definition of $n$-fold brackets in $\mathrm{TOP}^*$. 
In Appendix~A, we give a proof of Lemma 3.1. 
In Appendix~B, we write diagrams which shall be needed in Section 3. 
In Appendix~C, we correct errors in \cite{OO1,OO2}. 

\section{Review of the definition of $\{\vec{\bm f}\,\}^{(\ddot{s}_t)}_{\vec{\bm m}}$}
Let $\mathrm{TOP}$ be the category of topological spaces (spaces for short) 
and continuous maps (maps for short), $\mathrm{TOP}^*$ the category of spaces 
with the base point (pointed spaces for short) and maps preserving the base point 
(pointed maps for short), and $\mathrm{TOP}^w$ the full subcategory of 
$\mathrm{TOP}^*$ of well-pointed spaces, that is, pointed spaces such that 
the inclusion of the base point 
to the pointed space is a cofibration in $\mathrm{TOP}$. 
 
Given a space $A$, let $\mathrm{TOP}^A$ be the category of spaces under $A$ 
and maps under $A$. 
That is, its objects consists of maps $j:A\to X$ and $\mathrm{TOP}^A(j,j')$, 
the set of morphisms from $j:A\to X$ to $j':A\to X'$ consists of maps 
$f:X\to X'$ such that $f\circ j=j'$. 
The following well-known result of Dold (\cite[(2.18)]{DKP}) is useful: if $j:A\to X$ and $j':A\to X'$ 
are cofibrations and $f\in\mathrm{TOP}^A(j,j')$ is a homotopy equivalence in 
$\mathrm{TOP}$, then $f$ is a homotopy equivalence in $\mathrm{TOP}^A$. 
In the last case, we denote by $f^{-1}\in\mathrm{TOP}^A(j',j)$ 
a homotopy inverse of $f$ in $\mathrm{TOP}^A$. 

We list up some notations in $\mathrm{TOP}^*$. 
The base point of the pointed space $X$ is denoted by $x_0$ or $*$. 
Let $f:X\to Y$ be a pointed map. 
\begin{align*}
&I=[0,1],\ \text{the closed unit interval having $1$ as its base point};\\
&X_1\wedge\cdots\wedge X_n=(X_1\times\cdots\times X_n)/T(X_1,\dots,X_n),\\
&\hspace{1cm} \text{$T(X_1,\dots,X_n)$ is the subspace of $X_1\times\cdots\times X_n$ consisting of points such that}\\
&\hspace{1cm} \text{at least one component is the base point};\\
& CX=X\wedge I,\  \Sigma X=CX/(X\wedge\{0,1\});\\ 
&\s^m=\{(t_1,\dots,t_{m+1})\in\mathbf{R}^{m+1}\,|\,\sum t_k^2=1\},\\
&\hspace{1cm} \text{the $m$-sphere having $(1,0,\dots,0)$ as its base point},\\
&\hspace{1cm}\text{we identify $\s^m$ with $\underbrace{\s^1\wedge\cdots\wedge\s^1}_m=(\s^1)^{\wedge m}$ for $m\ge 1$ as in Section 2 of \cite{OO2}};\\
&1_X:X\to X,\ \text{the identity map of the space $X$};\\
&\Sigma^m X=X\wedge\s^m,\ \text{we identify $\Sigma X$ with $\Sigma^1X$};\\
&\Sigma^mf=f\wedge 1_{\s^m}:X\wedge\s^m\to Y\wedge\s^m;\\
&\s^{m_{[s+r,s]}}=\s^{m_s}\wedge\cdots\wedge\s^{m_{s+r}},\ 
\Sigma^{m_{[s+r,s]}}X=X\wedge\s^{m_{[s+r,s]}}, \Sigma^{m_{[s+r,s]}}f=f\wedge 1_{\s^{m_{[s+r,s]}}},\\
&\hspace{2cm} \text{for non-negative integers $r$ and $m_{s+r},\dots,m_s$};\\ 
&1_f:X\times I\to Y,\ \text{the constant homotopy $(x,t)\mapsto f(x)$};\\
&\widetilde{\Sigma}^mH:\Sigma^m X\times I\to \Sigma^m Y,\ (x\wedge s,t)\mapsto H(x,t)\wedge s,\ \text{for a homotopy $H:X\times I\to Y$};\\
&C_f=Y\cup_f CX, \ \text{the mapping cone of $f$};\\
&\psi^m_f:\Sigma^mY\cup_{\Sigma^mf}C\Sigma^mX
\underset{\approx}{\longrightarrow}\Sigma^m(Y\cup_f CX),\ y\wedge s\mapsto y\wedge s,\ x\wedge s\wedge t\mapsto x\wedge t\wedge s;\\
& i_f:Y\to Y\cup_f CX,\ \text{the inclusion};\\
& q_f:Y\cup_f CX\to (Y\cup_f CX)/Y=\Sigma X,\ \text{the quotient};\\
&q'_f:(Y\cup_f CX)\cup_{i_f}CY\underset{\simeq}{\longrightarrow} ((Y\cup_f CX)\cup_{i_f} CY)/CY=\Sigma X,\ \text{the quotient};\\
&\tau(X,Y):X\wedge Y\to Y\wedge X,\ \text{the switching map},\ x\wedge y\mapsto y\wedge x.
\end{align*}
Given a pointed map $g:Y\to Z$ with a null homotopy $A:CX\to Z$ of $g\circ f$, we define
\begin{gather*}
[g,A,f]:Y\cup_f CX\to Z,\ y\mapsto g(y),\ x\wedge t\mapsto A(x\wedge t);\\
(g,A,f):\Sigma X\to Z\cup_g CY,\ x\wedge t\mapsto \begin{cases} f(x)\wedge(1-2t) & 0\le t\le 1/2\\
A(x\wedge (2t-1)) & 1/2\le t\le 1\end{cases}.
\end{gather*}
Given a homotopy commutative square and a homotopy
\begin{equation}
\xymatrix{
X\ar[d]^-a \ar[r]^-f & Y\ar[d]^-b\\
X'\ar[r]^-{f'} & Y' 
}
,\quad 
H:b\circ f\simeq f'\circ a,
\end{equation}
we define $\Phi(f,f',a,b;H):Y\cup_f CX\to Y'\cup_{f'}CX'$ by  
$$
y\mapsto b(y),\quad x\wedge t\mapsto \begin{cases} H(x,2t) & 0\le t\le 1/2\\
a(x)\wedge(2t-1) & 1/2\le t\le 1\end{cases}.
$$ 

\begin{lemma}[\S2 of \cite{P}]
Given (2.1), consider the following diagram:
$$
\xymatrix{
X\ar[d]^-a \ar[r]^-f & Y\ar[d]^-b \ar[r]^-{i_f} & Y\cup_f CX \ar[d]^-{\Phi(H)} \ar[r]^-{i_{i_f}} & (Y\cup_f CX)\cup_{i_f} CY \ar[d]^-{\Phi'(H)} \ar[r]^-{q'_f} & \Sigma X\ar[d]^-{\Sigma a}\\
X'\ar[r]^-{f'} & Y' \ar[r]^-{i_{f'}} & Y'\cup_{f'}CX' \ar[r]^-{i_{i_{f'}}} & (Y'\cup_{f'}CX')\cup_{i_{f'}}CY' \ar[r]^-{q'_{f'}} & \Sigma X'
}
$$
where $\Phi(H)=\Phi(f,f',a,b;H)$ and 
$\Phi'(H)=\Phi(i_f,i_{f'},b,\Phi(H);1_{i_{f'}\circ b})$. 
Then we have
\begin{enumerate}
\item The second and the third squares are commutative; 
the first and the fourth squares are homotopy commutative.
\item $\Phi'(H)\simeq \Phi(H)\cup Cb$. 
\item $q'_{f'}\circ(\Phi(H)\cup Cb)\simeq \Sigma a\circ q'_f$. 
\item If $a$ and $b$ are homotopy equivalences, then $\Phi(H)$ is a homotopy equivalence. 
\item If the first square is strictly commutative and $H=1_{b\circ f}$, then $\Phi(H)\simeq b\cup Ca$. 
\end{enumerate}
\end{lemma}

From now on, we will work in $\mathrm{TOP}^w$ except Section 5. 

Suppose that (1.1) is given. 

In order to define the bracket $\{\vec{\bm f}\,\}^{(\ddot{s}_t)}_{\vec{\bm m}}$ we consider a collection $\{\mathscr{S}_r,\overline{f_r},\mathscr{A}_r\,|\,2\le r\le n\}$ which satisfies four conditions (i)--(iv): 
\begin{enumerate}
\item[(i)] $\mathscr{S}_r$ is a diagram displayed in 
$$
\footnotesize{
\xymatrix{
\Sigma^{m_{r-1}}X_{r-1}\ar[d]^-{f_{r-1}=g_{r,1}} &\cdots &\Sigma^{s-1}\Sigma^{m_{[r-1,r-s]}}X_{r-s} \ar[d]^-{g_{r,s}} &\cdots &\Sigma^{r-2}\Sigma^{m_{[r-1,1]}}X_1 \ar[d]^-{g_{r,r-1}}&\\
X_r=C_{r,1}\ar[r]_-{j_{r,1}} &\cdots \ar[r]_{j_{r,s-1}} & C_{r,s}\ar[r]_-{j_{r,s}} &\cdots\ar[r]_{j_{r,r-2}} & C_{r,r-1}\ar[r]_-{j_{r,r-1}} & C_{r,r}
}
}
$$
such that 
\begin{enumerate}
\item[(i.1)] it is {\it reduced}, that is, $C_{r,2}=X_r\cup_{f_{r-1}}C\Sigma^{m_{r-1}}X_{r-1}$ and $j_{r,1}=i_{f_{r-1}}$, 
\item[(i.2)] the pointed map $j_{r,s}:C_{r,s}\to C_{r,s+1}$ is a free cofibration, that is, $j_{r,s}$ is a cofibration in $\mathrm{TOP}$ and so $\Sigma^\ell j_{r,s}$ is also a free cofibration by \cite[Corollary 2.3(1)]{OO2} for any non-negative integer $\ell$ so that $\Sigma^\ell j_{r,s}$ is an embedding by Str{\o}m \cite[Theorem 1]{S},
\item[(i.3)] there is a homotopy equivalence in $\mathrm{TOP}^{C_{r,s}}(j_{r,s},i_{g_{r,s}})$ for $1\le s<r$.
\end{enumerate}
(Sometimes we write $\mathscr{S}_2=(\Sigma^{m_1}X_1;X_2, X_2\cup_{f_1}C\Sigma^{m_1}X_1;f_1;i_{f_1})$.) 
\item[(ii)] $\overline{f_r}:\begin{cases} \Sigma^{m_r}C_{r,r}\to X_{r+1} & 2\le r< n\\
\Sigma^{m_n}C_{n,n-1}\to X_{n+1} & r=n\end{cases}$ 
is an extension of $f_r$. 
Let $\overline{f_r}^s:\Sigma^{m_r}C_{r,s}\to X_{r+1}$ be the restriction of $\overline{f_r}$ to $\Sigma^{m_r}C_{r,s}$ for $1\le s\le r$ if $r< n$ and $1\le s< n$ if $r=n$ (so that 
$\overline{f_r}^1=f_r$). 
\item[(iii)] $\mathscr{A}_r=\{a_{r,s}\,|\,1\le s<r\}$ where $a_{r,s}\in \mathrm{TOP}^{C_{r,s}}(j_{r,s},i_{g_{r,s}})$ is a homotopy equivalence 
such that $\mathscr{A}_r$ is {\it reduced}, that is, $a_{r,1}=1_{C_{r,2}}$.  
It is called a {\it structure} on $\mathscr{S}_r$. 
\item[(iv)] $\mathscr{S}_{r+1}=(\widetilde{\Sigma}^{m_r}\mathscr{S}_r)(\overline{f_r},\widetilde{\Sigma}^{m_r}\mathscr{A}_r)\ (2\le r<n)$. 
\end{enumerate}
In the above list, (iv) is the only one which must be explained. 
In order to explain (iv), we need some terminology. 
The {\it quasi-structure} $\Omega(\mathscr{A}_r)$ derived from $\mathscr{A}_r$ on $\mathscr{S}_r$ consists of homotopy equivalences
$$
\omega_{r,s}=q'_{g_{r,s}}\circ(a_{r,s}\cup C1_{C_{r,s}}):C_{r,s+1}\cup_{j_{r,s}}CC_{r,s}\underset{\simeq}{\to}\Sigma\Sigma^{s-1}\Sigma^{m_{[r-1,r-s]}}X_{r-s}\ (1\le s< r).
$$
For $1\le s< r$, we set
\begin{align}
\widetilde{\Sigma}^{m_r}\omega_{r,s}&=(1_{\Sigma^{m_{[r-1,r-s]}}X_{r-s}}\wedge\tau(\s^{s-1}\wedge\s^1,\s^{m_r}))\circ\Sigma^{m_r}\omega_{r,s}\circ\psi^{m_r}_{j_{r,s}}\\
&\hspace{1cm}:\Sigma^{m_r}C_{r,s+1}\cup_{\Sigma^{m_r}j_{r,s}}C\Sigma^{m_r}C_{r,s}\underset{\simeq}{\longrightarrow}\Sigma\Sigma^{s-1}\Sigma^{m_{[r,r-s]}}X_{r-s},\nonumber\\
\widetilde{\Sigma}^{m_r}g_{r,s}&=\Sigma^{m_r}g_{r,s}\circ(1_{\Sigma^{m_{[r-1,r-s]}}X_{r-s}}\wedge\tau(\s^{m_r},\s^{s-1}))\\
&\hspace{1cm}:\Sigma^{s-1}\Sigma^{m_{[r,r-s]}}X_{r-s}\longrightarrow\Sigma^{m_r}C_{r,s},\nonumber\\
\widetilde{\Sigma}^{m_r}a_{r,s}&=(1_{\Sigma^{m_r}C_{r,s}}\cup C(1_{\Sigma^{m_{[r-1,r-s]}}X_{r-s}}\wedge\tau(\s^{s-1},\s^{m_r})))\circ(\psi^{m_r}_{g_{r,s}})^{-1}\circ\Sigma^{m_r}a_{r,s}\\
&\hspace{1cm}:\Sigma^{m_r}C_{r,s+1}\underset{\simeq}{\longrightarrow}\Sigma^{m_r}C_{r,s}\cup_{\widetilde{\Sigma}^{m_r}g_{r,s}}C\Sigma^{s-1}\Sigma^{m_{[r,r-s]}}X_{r-s}.\nonumber
\end{align}
Set $\widetilde{\Sigma}^{m_r}\mathscr{A}_r=\{\widetilde{\Sigma}^{m_r}a_{r,s}\,|\,1\le s< r\}$. 
We easily have
\begin{equation}
\widetilde{\Sigma}^{m_r}\omega_{r,s}=q'_{\widetilde{\Sigma}^{m_r}g_{r,s}}\circ(\widetilde{\Sigma}^{m_r}a_{r,s}\cup C1_{\Sigma^{m_r}C_{r,s}})\quad(1\le s< r).
\end{equation}
Let $\widetilde{\Sigma}^{m_r}\mathscr{S}_r$ be the following diagram:
$$
\tiny{
\xymatrix{
\Sigma^{m_{[r,r-1]}}X_{r-1}\ar[d]^-{\Sigma^{m_r}f_{r-1}=\widetilde{\Sigma}^{m_r}g_{r,1}} &\cdots & \Sigma^{s-1}\Sigma^{m_{[r,r-s]}}X_{r-s}\ar[d]^-{\widetilde{\Sigma}^{m_r}g_{r,s}} &\cdots &\Sigma^{r-2}\Sigma^{m_{[r,1]}}X_1 \ar[d]^-{\widetilde{\Sigma}^{m_r}g_{r,r-1}}&\\
\Sigma^{m_r}X_r=\Sigma^{m_r}C_{r,1}\ar[r]_-{\Sigma^{m_r}j_{r,1}} &\cdots\ar[r]&\Sigma^{m_r}C_{r,s}\ar[r]_-{\Sigma^{m_r}j_{r,s}}&\cdots\ar[r]&\Sigma^{m_r}C_{r,r-1}\ar[r]_-{\Sigma^{m_r}j_{r,r-1}}&\Sigma^{m_r}C_{r,r}
}
}
$$
We see that $\widetilde{\Sigma}^{m_r}a_{r,s}\in\mathrm{TOP}^{\Sigma^{m_r}C_{r,s}}(\Sigma^{m_r}j_{r,s},i_{\widetilde{\Sigma}^{m_r}g_{r,s}})$ is a homotopy equivalence in the category $\mathrm{TOP}^{\Sigma^{m_r}C_{r,s}}$. 
Let 
\begin{equation}
(\widetilde{\Sigma}^{m_r}a_{r,s})^{-1}\in\mathrm{TOP}^{\Sigma^{m_r}C_{r,s}}(i_{\widetilde{\Sigma}^{m_r}g_{r,s}},\Sigma^{m_r}j_{r,s})
\end{equation}
be a homotopy inverse of 
$\widetilde{\Sigma}^{m_r}a_{r,s}$. 
It follows from the proof of \cite[Lemma 4.3(1)]{OO2} that 
$
(\widetilde{\Sigma}^{m_r}a_{r,s})^{-1}\cup C1_{\Sigma^{m_r}C_{r,s}}
:(\Sigma^{m_r}C_{r,s}\cup_{\widetilde{\Sigma}^{m_r}}C\Sigma^{s-1}\Sigma^{m_{[r,r-s]}}X_{r-s})
\cup_{i_{\widetilde{\Sigma}^{m_r}g_{r,s}}}C\Sigma^{m_r}C_{r,s}\to\Sigma^{m_r}C_{r,s+1}\cup_{\Sigma^{m_r}j_{r,s}}C\Sigma^{m_r}C_{r,s}
$ is a homotopy inverse of $\widetilde{\Sigma}^{m_r}a_{r,s}\cup C1_{\Sigma^{m_r}C_{r,s}}$ so that
\begin{equation}
\left\{\begin{array}{@{\hspace{0.2mm}}ll}
&(\widetilde{\Sigma}^{m_r}\omega_{r,s})^{-1}\simeq((\widetilde{\Sigma}^{m_r}a_{r,s})^{-1}\cup C1_{\Sigma^{m_r}C_{r,s}})\circ(q'_{\widetilde{\Sigma}^{m_r}g_{r,s}})^{-1},\\
&\widetilde{\Sigma}^{m_r}\omega_{r,s}\circ ((\widetilde{\Sigma}^{m_r}a_{r,s})^{-1}\cup C1_{\Sigma^{m_r}C_{r,s}})\simeq q'_{\widetilde{\Sigma}^{m_r}g_{r,s}}
\end{array}\right.
\end{equation}
by (2.5), where $(\widetilde{\Sigma}^{m_r}\omega_{r,s})^{-1}$ is a homotopy inverse of $\widetilde{\Sigma}^{m_r}\omega_{r,s}$. 
We denote the following diagram for $2\le r<n$ by $(\widetilde{\Sigma}^{m_r}\mathscr{S}_r)(\overline{f_r},\widetilde{\Sigma}^{m_r}\mathscr{A}_r)$:
$$
\xymatrix{
\Sigma^{m_r}X_r\ar[d]^-{g'_{r+1,1}} &\cdots &\Sigma^s\Sigma^{m_{[r,r-s]}}X_{r-s} \ar[d]^-{g'_{r+1,s+1}}&\cdots&\Sigma^{r-1}\Sigma^{m_{[r,1]}}X_1\ar[d]^-{g'_{r+1,r}} &\\
C'_{r+1,1}\ar[r]^-{j'_{r+1,1}}&\cdots\ar[r]^-{j'_{r+1,s}} &C'_{r+1,s+1}\ar[r]^-{j'_{r+1,s+1}}&\cdots\ar[r]^-{j'_{r+1,r-1}}&C'_{r+1,r}\ar[r]^-{j'_{r+1,r}}&C'_{r+1,r+1}
}
$$ 
where
\begin{gather*}
C'_{r+1,1}=X_{r+1},\quad C'_{r+1,s+1}=X_{r+1}\cup_{\overline{f_r}^s}C\Sigma^{m_r}C_{r,s}\ (1\le s\le r),\\
j'_{r+1,1}=i_{f_r},\quad j'_{r+1,s+1}=1_{X_{r+1}}\cup C\Sigma^{m_r}j_{r,s}\ (1\le s< r),\\
g'_{r+1,1}=f_r,\quad g'_{r+1,s+1}=(\overline{f_r}^{s+1}\cup C1_{\Sigma^{m_r}C_{r,s}})\circ(\widetilde{\Sigma}^{m_r}\omega_{r,s})^{-1}\ (1\le s< r).
\end{gather*}
Now (iv) is interpreted as 
$$
C_{r+1,s}=C'_{r+1,s}\ (1\le s\le r+1),\quad 
j_{r+1,s}=j_{r+1,s}'\ (1\le s\le r),\quad g_{r+1,s}=g_{r+1,s}'\ (1\le s\le r).
$$ 
We must confirm that $(\widetilde{\Sigma}^{m_r}\mathscr{S}_r)(\overline{f_r},\widetilde{\Sigma}^{m_r}\mathscr{A}_r)$ has necessary properties. 
By \cite[Proposition~2.2, Corollary 2.3]{OO2} 
$C'_{r+1,s+1}$ is well-pointed and $j'_{r+1,s}$ is a free cofibration for $1\le s\le r$. 
It follows from \cite[Lemma 5.3]{OO2} (cf.\,\cite[Section 3]{OO3}) that 
$(\widetilde{\Sigma}^{m_r}\mathscr{S}_r)(\overline{f_r},\widetilde{\Sigma}^{m_r}\mathscr{A}_r)$ 
has a reduced structure, that is, there exist homotopy equivalences 
$a'_{r+1,s}\in\mathrm{TOP}^{C'_{r+1,s}}(j'_{r+1,s},i_{g'_{r+1,s}})$ for $1\le s\le r$ with $a'_{r+1,1}=1_{C'_{r+1,2}}$. 
This ends the explanation of (iv). 

We call a collection $\{\mathscr{S}_r,\overline{f_r},\mathscr{A}_r\,|\,2\le r\le n\}$ an $\ddot{s}_t$-{\it presentation} of $\vec{\bm f}$ if it satisfies (i)--(iv). 
We denote by $\{\vec{\bm f}\,\}^{(\ddot{s}_t)}_{\vec{\bm m}}$ the set of homotopy classes of 
$$
\overline{f_n}\circ\widetilde{\Sigma}^{m_n}g_{n,n-1}:\Sigma^{n-2}\Sigma^{m_{[n,1]}}X_1\to X_{n+1}
$$
for all $\ddot{s}_t$-presentations $\{\mathscr{S}_r,\overline{f_r},\mathscr{A}_r\,|\,2\le r\le n\}$ of $\vec{\bm f}$. 
As seen in \cite[Theorem 6.1]{OO3}, $\{\vec{\bm f}\}^{(\ddot{s}_t)}_{\vec{\bm m}}$ depends only on the homotopy classes of $f_k\ (1\le k\le n)$. 

Suppose that $n\ge 4$ and an $\ddot{s}_t$-presentation $\{\mathscr{S}_r,\overline{f_r},\mathscr{A}_r\,|\,2\le r\le n\}$ of $\vec{\bm f}$ is given. 
Since $\overline{f_{r+1}}^2\circ\psi^{m_{r+1}}_{f_r}:\Sigma^{m_{r+1}}X_{r+1}\cup_{\Sigma^{m_{r+1}}f_r}C\Sigma^{m_{[r+1,r]}}X_r\to X_{r+2}$ is 
an extension of $f_{r+1}$, we can take $A_r:f_{r+1}\circ\Sigma^{m_{r+1}}f_r\simeq *\ (1\le r<n)$ such that $\overline{f_{r+1}}^2\circ\psi^{m_{r+1}}_{f_r}=[f_{r+1},A_r,\Sigma^{m_{r+1}}f_r]$. 
Thus we have a sequence $(A_{n-1},\dots,A_1)$ of null homotopies $A_r:f_{r+1}\circ\Sigma^{m_{r+1}}f_r\simeq *$ such that 
$\overline{f_{r+1}}^2\circ\psi^{m_{r+1}}_{f_r}=[f_{r+1},A_r,\Sigma^{m_{r+1}}f_r]\ (1\le r< n)$ 
so that 
$g_{r+2,2}\simeq (f_{r+1},A_r,\Sigma^{m_{r+1}}f_r)\ (1\le r\le n-2)$ by \cite[(4.2)]{OO2}. 
By (2.6), for $2\le r<n$, 
$$
\overline{f_{r+1}}^3\circ(\widetilde{\Sigma}^{m_{r+1}}a_{r+1,2})^{-1}:\Sigma^{m_{r+1}}C_{r+1,2}\cup_{\widetilde{\Sigma}^{m_{r+1}}g_{r+1,2}}C\Sigma\Sigma^{m_{[r+1,r-1]}}X_{r-1}\to X_{r+2}
$$
is an extension of $\overline{f_{r+1}}^2$ so that $\overline{f_{r+1}}^2\circ\widetilde{\Sigma}^{m_{r+1}}g_{r+1,2}\simeq *$. 
We have
\begin{align*}
* &\simeq \overline{f_{r+1}}^2\circ\widetilde{\Sigma}^{m_{r+1}}g_{r+1,2}\\
&=\overline{f_{r+1}}^2\circ\psi^{m_{r+1}}_{f_r}\circ(\psi^{m_{r+1}}_{f_r})^{-1}\circ\Sigma^{m_{r+1}}g_{r+1,2}\circ(1_{\Sigma^{m_{[r,r-1]}}X_{r-1}}\wedge\tau(\s^{m_{r+1}},\s^1))\\
&\simeq [f_{r+1},A_r,\Sigma^{m_{r+1}}f_r]\circ (\psi^{m_{r+1}}_{f_r})^{-1}\circ\Sigma^{m_{r+1}}(f_r,A_{r-1},\Sigma^{m_r}f_{r-1})\\
&\hspace{4cm}\circ(1_{\Sigma^{m_{[r,r-1]}}X_{r-1}}\wedge\tau(\s^{m_{r+1}},\s^1))\\
&=[f_{r+1},A_r,\Sigma^{m_{r+1}}f_r]\circ (\Sigma^{m_{r+1}}f_r,\widetilde{\Sigma}^{m_{r+1}}A_{r-1},\Sigma^{m_{[r+1,r]}}f_{r-1})\ (\text{by \cite[Lemma 2.4]{OO1}}).
\end{align*}
Therefore $\vec{\bm A}=(A_{n-1},\dots,A_1)$ is an admissible sequence of null homotopies for $\vec{\bm f}$. 
Hence we have

\begin{prop}
If $n\ge 4$ and $\{\vec{\bm f}\}^{(\ddot{s}_t)}_{\vec{\bm m}}$ is not empty, then for any element $\alpha$ of $\{\vec{\bm f}\}^{(\ddot{s}_t)}_{\vec{\bm m}}$ there exist an $\ddot{s}_t$-presentation $\{\mathscr{S}_r,\overline{f_r},\mathscr{A}_r\,|\,2\le r\le n\}$ 
of $\vec{\bm f}$ and an admissible sequence $\vec{\bm A}=(A_{n-1},\dots,A_1)$ 
of null homotopies 
for $\vec{\bm f}=(f_n,\dots,f_1)$ such that $\alpha=\overline{f_n}\circ\widetilde{\Sigma}^{m_n}g_{n,n-1}$ and 
\begin{gather*}
\overline{f_r}^2\circ\psi^{m_r}_{f_{r-1}}=[f_r,A_{r-1},\Sigma^{m_r}f_{r-1}]\ (2\le r\le n),\\
g_{r+1,2}\simeq (f_r,A_{r-1},\Sigma^{m_r}f_{r-1})\ (2\le r< n).
\end{gather*}
\end{prop}

\section{Proof of Theorem 1.1}
By Lemma A.1(2) in Appendix A, we have

\begin{lemma}
\begin{align*}
&\{f_n,\dots,f_4,[f_3,A_2,\Sigma^{m_3}f_2],(\Sigma^{m_3}f_2,\widetilde{\Sigma}^{m_3}A_1,\Sigma^{m_{[3,2]}}f_1)\}^{(\ddot{s}_t)}_{(m_n,\dots,m_4,0,0)}\\
&=\{f_n,\dots,f_4,[f_3,A_2,\Sigma^{m_3}f_2]\circ(\psi^{m_3}_{f_2})^{-1},\psi^{m_3}_{f_2}\circ(\Sigma^{m_3}f_2,\widetilde{\Sigma}^{m_3}A_1,\Sigma^{m_{[3,2]}}f_1)\}^{(\ddot{s}_t)}_{(m_n,\dots,m_4,0,0)}.
\end{align*}
\end{lemma}

Thus, since $\bigcup_{\text{admissible }\vec{\bm A}}\subset\bigcup_{{\text all }\vec{\bm A}}\subset\bigcup_{A_2,A_1}$, 
it suffices to prove the following two containments for our purpose.
\begin{align}
\begin{split}
&\{f_n,\dots,f_1\}^{(\ddot{s}_t)}_{\vec{\bm m}}\\
&\quad\subset\bigcup_{{\text admissible }\vec{\bm A}}\{f_n,\dots,f_4,[f_3,A_2,\Sigma^{m_3}f_2]\circ(\psi^{m_3}_{f_2})^{-1},\\
&\hspace{2cm}\psi^{m_3}_{f_2}\circ(\Sigma^{m_3}f_2,\widetilde{\Sigma}^{m_3}A_1,\Sigma^{m_{[3,2]}}f_1)\}^{(\ddot{s}_t)}_{(m_n,\dots,m_4,0,0)}\\
&\hspace{2.5cm}\circ
(1_{\Sigma^{m_{[3,1]}}X_1}\wedge\tau(\s^{m_{[n,4]}},\s^1)\wedge 1_{(\s^1)^{\wedge(n-3)}}),  
\end{split}\\
\begin{split}
&\{f_n,\dots,f_1\}^{(\ddot{s}_t)}_{\vec{\bm m}}\\
&\quad\supset\bigcup_{A_2,A_1}\{f_n,\dots,f_4,[f_3,A_2,\Sigma^{m_3}f_2]\circ(\psi^{m_3}_{f_2})^{-1},\\
&\hspace{2cm}\psi^{m_3}_{f_2}\circ(\Sigma^{m_3}f_2,\widetilde{\Sigma}^{m_3}A_1,\Sigma^{m_{[3,2]}}f_1)\}^{(\ddot{s}_t)}_{(m_n,\dots,m_4,0,0)}\\
&\hspace{2.5cm}\circ(1_{\Sigma^{m_{[3,1]}}X_1}\wedge\tau(\s^{m_{[n,4]}},\s^1)\wedge 1_{(\s^1)^{\wedge(n-3)}}).  
\end{split}
\end{align}

This section consists of two subsections \S3.1 and \S3.2 in which we will prove (3.1) and (3.2), respectively. 

In \S3.1 and \S3.2, we will use the following notations:
\begin{gather*}
X^*_1=\Sigma\Sigma^{m_{[3,1]}}X_1,\quad X^*_2=\Sigma^{m_3}(X_3\cup_{f_2}C\Sigma^{m_2}X_2),\quad X^*_k=X_{k+1}\ (3\le k\le n),\\
m^*_1=m^*_2=0,\quad m^*_k=m_{k+1}\ (3\le k< n).
\end{gather*}

\numberwithin{equation}{subsection}

\subsection{Proof of (3.1)}
Let $n\ge 4$. 
It suffices to prove (3.1) when $\{\vec{\bm f}\,\}^{(\ddot{s}_t)}_{\vec{\bm m}}$ is not empty. 
Take $\alpha\in\{f_n,\dots,f_1\}^{(\ddot{s}_t)}_{\vec{\bm m}}$. 
By Proposition 2.2, there exist an $\ddot{s}_t$-presentation $\{\mathscr{S}_r,\overline{f_r},\mathscr{A}_r\,|\,2\le r\le n\}$ of $\vec{\bm f}$ 
and an admissible sequence $\vec{\bm A}=(A_{n-1},\dots,A_1)$ of null homotopies $A_k:f_{k+1}\circ\Sigma^{m_{k+1}}f_k\simeq *$ 
such that $\alpha=\overline{f_n}\circ\widetilde{\Sigma}^{m_n}g_{n,n-1}$ and
\begin{gather*}
\overline{f_r}^2\circ\psi^{m_r}_{f_{r-1}}=[f_r,A_{r-1},\Sigma^{m_r}f_{r-1}]\ (2\le r\le n),\\
g_{r+1,2}\simeq(f_r,A_{r-1},\Sigma^{m_r}f_{r-1})\ (2\le r< n).
\end{gather*}
Set
\begin{gather*}
f^*_1=\widetilde{\Sigma}^{m_3}g_{3,2}:X^*_1\to X^*_2,\quad f^*_2=\overline{f_3}^2:X^*_2\to X^*_3,\\
f^*_k=f_{k+1}:\Sigma^{m^*_k}X^*_k\to X^*_{k+1}\ (3\le k< n).
\end{gather*}
Then $f^*_2=[f_3,A_2,\Sigma^{m_3}f_2]\circ(\psi^{m_3}_{f_2})^{-1}$ and 
\begin{align*}
f^*_1&=\widetilde{\Sigma}^{m_3}g_{3,2}=\Sigma^{m_3}g_{3,2}\circ(1_{\Sigma^{m_{[2,1]}}X_1}\wedge\tau(\s^{m_3},\s^1))\\
&\simeq\Sigma^{m_3}(f_2,A_1,\Sigma^{m_2}f_1)\circ(1_{\Sigma^{m_{[2,1]}}X_1}\wedge\tau(\s^{m_3},\s^1))\\
&=\psi^{m_3}_{f_2}\circ(\Sigma^{m_3}f_2,\widetilde{\Sigma}^{m_3}A_1,\Sigma^{m_{[3,2]}}f_1)\quad(\text{by \cite[Lemma 2.4]{OO1}}).
\end{align*}
Hence by \cite[Theorem 6.1]{OO3} we have
\begin{align*}
&\{f^*_{n-1},\dots,f^*_3,f^*_2,f^*_1\}^{(\ddot{s}_t)}_{(m^*_{n-1},\dots,m^*_3,0,0)}\\
&=\{f_n,\dots,f_4,[f_3,A_2,\Sigma^{m_3}f_2]\circ(\psi^{m_3}_{f_2})^{-1},\psi^{m_3}_{f_2}\circ(\Sigma^{m_3}f_2,\widetilde{\Sigma}^{m_3}A_1,\Sigma^{m_{[3,2]}}f_1)\}^{(\ddot{s}_t)}_{(m_n,\dots,m_4,0,0)}.
\end{align*}
We will construct an $\ddot{s}_t$-presentation $\{\mathscr{S}^*_r,\overline{f^*_r},\mathscr{A}^*_r\,|\,2\le r< n\}$ of $\vec{\bm f^*}$ 
such that 
$\overline{f^*_{n-1}}=\overline{f_n}$ and 
$$
\widetilde{\Sigma}^{m_n}g_{n,n-1}
\simeq \widetilde{\Sigma}^{m_{n-1}^*}g^*_{n-1,n-2}\circ(1_{\Sigma^{m_{[3,1]}}X_1}\wedge\tau(\s^{m_{[n,4]}},\s^1)\wedge
 1_{(\s^1)^{\wedge(n-3)}}).
$$
Once we have such an $\ddot{s}_t$-presentation, then
$$
\overline{f_n}\circ\widetilde{\Sigma}^{m_n}g_{n,n-1}\simeq\overline{f^*_{n-1}}\circ\widetilde{\Sigma}^{m^*_{n-1}}g^*_{n-1,n-2}\circ(1_{\Sigma^{m_{[3,1]}}X_1}\wedge\tau(\s^{m_{[n,4]}},\s^1)\wedge
 1_{(\s^1)^{\wedge(n-3)}})
$$ 
so that we obtain (3.1). 

First we set $\mathscr{S}_2^*=(X^*_1;X^*_2,X^*_2\cup_{f^*_1}CX^*_1;f^*_1;i_{f^*_1})$ and 
$\mathscr{A}^*_2=\{1_{C^*_{2,2}}\}$. 
Then 
$$
C^*_{2,1}=\Sigma^{m_3}C_{3,2}, \quad C^*_{2,2}=\Sigma^{m_3}C_{3,2}\cup_{\widetilde{\Sigma}^{m_3}g_{3,2}}C\Sigma\Sigma^{m_{[3,1]}}X_1,
$$ 
$\widetilde{\Sigma}^{m_3}a_{3,2}:\Sigma^{m_3}C_{3,3}\to C^*_{2,2}$ is a homotopy equivalence in $\mathrm{TOP}^{\Sigma^{m_3}C_{3,2}}(\Sigma^{m_3}j_{3,2},i_{\widetilde{\Sigma}^{m_3}g_{3,2}})$, and 
\begin{equation}
\omega^*_{2,1}=q'_{\widetilde{\Sigma}^{m_3}g_{3,2}}.
\end{equation}
Set $\overline{f_2^*}=\overline{f_3}\circ(\widetilde{\Sigma}^{m_3}a_{3,2})^{-1}:C^*_{2,2}\to X^*_3$. 

Secondly set $\mathscr{S}^*_3=\mathscr{S}^*_2(\overline{f^*_2},\mathscr{A}^*_2)$. 
Then $g^*_{3,1}=f^*_2$ and
$$
C^*_{3,1}=C_{4,1},\quad C^*_{3,2}=C_{4,3},\quad C^*_{3,3}=X^*_3\cup_{\overline{f^*_2}}CC^*_{2,2}.
$$
Set 
$$
\overline{f^*_3}=\begin{cases} \overline{f_4} :\Sigma^{m^*_3}C^*_{3,2}\to X^*_4 & n=4\\
\overline{f_4}\circ\Sigma^{m^*_3}(1_{X_4}\cup C(\widetilde{\Sigma}^{m_3}a_{3,2})^{-1}) : \Sigma^{m^*_3}C^*_{3,3}\to X^*_4 & n\ge 5\end{cases}.
$$
By (2.7) and (3.1.1), we have 
$$
(\widetilde{\Sigma}^{m_3}\omega_{3,2})^{-1}\simeq((\widetilde{\Sigma}^{m_3}a_{3,2})^{-1}\cup C1_{\Sigma^{m_3}C_{3,2}})\circ(\omega^*_{2,1})^{-1}
$$
and so 
\begin{align*}
g_{4,3}&=(\overline{f_3}\cup C1_{\Sigma^{m_3}C_{3,2}})\circ(\widetilde{\Sigma}^{m_3}\omega_{3,2})^{-1}\\
&\simeq (\overline{f_3}\cup C1_{\Sigma^{m_3}C_{3,2}})\circ((\widetilde{\Sigma}^{m_3}a_{3,2})^{-1}\cup C1_{\Sigma^{m_3}C_{3,2}})\circ(\omega^*_{2,1})^{-1}\\
&=(\overline{f_2^*}\cup C1_{X^*_2})\circ(\omega^*_{2,1})^{-1}=g^*_{3,2}.
\end{align*}
Hence 
\begin{equation}
\widetilde{\Sigma}^{m_4}g_{4,3}\simeq\widetilde{\Sigma}^{m^*_3}g^*_{3,2}\circ(1_{\Sigma^{m_{[3,1]}}X_1}\wedge\tau(\s^{m_4},\s^1)\wedge 1_{\s^1}).
\end{equation}
Let $H^4:g_{4,3}\simeq g^*_{3,2}$ be an any homotopy and set
\begin{align}
\Phi(H^4)&=\Phi(g_{4,3},g^*_{3,2},1_{\Sigma^2\Sigma^{m_{[3,1]}}X_1},1_{C_{4,3}};H^4)\nonumber\\
&\hspace{1cm}:C_{4,3}\cup_{g_{4,3}}C\Sigma^2\Sigma^{m_{[3,1]}}X_1\to C^*_{3,2}\cup_{g^*_{3,2}}C\Sigma\Sigma^{m^*_{[2,1]}} X^*_1,\nonumber\\
a^*_{3,2}&=\Phi(H^4)\circ a_{4,3}\circ(1_{X_4}\cup C(\widetilde{\Sigma}^{m_3}a_{3,2})^{-1})
:C^*_{3,3}\to C^*_{3,2}\cup_{g^*_{3,2}}C\Sigma\Sigma^{m^*_{[2,1]}} X^*_1.
\end{align}
Set $\mathscr{A}^*_3=\{1_{C^*_{3,2}}, a^*_{3,2}\}$ which is a reduced structure on $\mathscr{S}^*_3$. 

When $n=4$, $\{\mathscr{S}^*_r,\overline{f^*_r},\mathscr{A}^*_r\,|\,r=2,3\}$ is an $\ddot{s}_t$-presentation of $(f^*_3,f^*_2,f^*_1)$ such that 
$$
\overline{f_4}\circ \widetilde{\Sigma}^{m_4}g_{4,3}\simeq\overline{f^*_3}\circ\widetilde{\Sigma}^{m^*_3}g^*_{3,2}\circ (1_{\Sigma^{m_{[3,1]}}X_1}\wedge\tau(\s^{m_4},\s^1)\wedge 1_{\s^1})
$$
by (3.1.2). 
Hence (3.1) holds for $n=4$. 

Thirdly let $n\ge 5$ and set $\mathscr{S}_4^*=(\widetilde{\Sigma}^{m^*_3}\mathscr{S}^*_3)(\overline{f^*_3},\widetilde{\Sigma}^{m^*_3}\mathscr{A}_3^*)$. 
Then 
$$
C^*_{4,s}=C_{5,s}\ (1\le s\le 2),\quad C^*_{4,3}=C_{5,4},\quad C^*_{4,4}=X^*_4\cup_{\overline{f^*_3}} C\Sigma^{m^*_3}C^*_{3,3}.
$$
Set 
\begin{align*}
\overline{f_4^*}&=\begin{cases} \overline{f_5} :\Sigma^{m^*_4}C^*_{4,3}=\Sigma^{m_5}C_{5,4}\to X_6=X^*_5 & n=5\\
\overline{f_5}\circ\Sigma^{m^*_4}(1_{X_5}\cup C\Sigma^{m^*_3}(1_{X_4}\cup C(\widetilde{\Sigma}^{m_3}a_{3,2})^{-1})):\Sigma^{m^*_4}C^*_{4,4}\to X_5^* & n\ge 6\end{cases},\\
\Phi(\widetilde{\Sigma}^{m_4} H^4)&=\Phi(\Sigma^{m_4}g_{4,3},\Sigma^{m^*_3}g^*_{3,2},1_{\Sigma^{m_4}\Sigma^2\Sigma^{m_{[3,1]}}X_1},1_{\Sigma^{m_4}C_{4,3}};\widetilde{\Sigma}^{m_4}H^4)\\
&:\Sigma^{m_4}C_{4,3}\cup_{\Sigma^{m_4}g_{4,3}}C\Sigma^{m_4}\Sigma^2\Sigma^{m_{[3,1]}}X_1\to \Sigma^{m^*_3}C^*_{3,2}\cup_{\Sigma^{m^*_3}g^*_{3,2}}C\Sigma^{m^*_3}\Sigma\Sigma^{m^*_{[2,1]}} X^*_1,\\
\Phi'(\widetilde{\Sigma}^{m_4} H^4)&=\Phi(i_{\Sigma^{m_4}g_{4,3}},i_{\Sigma^{m^*_3}g^*_{3,2}},1_{\Sigma^{m_4}C_{4,3}},\Phi(\widetilde{\Sigma}^{m_4}H^4);1_{i_{\Sigma^{m^*_3}g^*_{3,2}}})\\
&:(\Sigma^{m_4}C_{4,3}\cup_{\Sigma^{m_4}g_{4,3}} C\Sigma^{m_4}\Sigma\Sigma\Sigma^{m_{[3,1]}}X_1)\cup C\Sigma^{m_4}C_{4,3}\\
&\hspace{2cm}\to 
(\Sigma^{m^*_3}C^*_{3,2}\cup_{\Sigma^{m^*_3}g^*_{3,2}} C\Sigma^{m^*_3}\Sigma\Sigma^{m^*_{[2,1]}}X^*_1)\cup C\Sigma^{m^*_3}C^*_{3,2},\\
\tau_4&=1_{\Sigma^{m_{[3,1]}}X_1}\wedge\tau(\s^{m_4},(\s^1)^{\wedge 2}):\Sigma\Sigma\Sigma^{m_{[4,1]}}X_1\to\Sigma^{m_4}\Sigma\Sigma\Sigma^{m_{[3,1]}}X_1,\\
\tau'_4&=1_{\Sigma^{m_4}\Sigma^2\Sigma^{m_{[3,1]}}X_1}:\Sigma^{m_4}\Sigma^2\Sigma^{m_{[3,1]}}X_1\to\Sigma^{m^*_3}\Sigma\Sigma^{m^*_{[2,1]}}X^*_1=\Sigma^{m_4}\Sigma^2\Sigma^{m_{[3,1]}}X_1,\\
\tau''_4&=1_{\Sigma\Sigma^{m_{[3,1]}}X_1}\wedge\tau(\s^1,\s^{m_4}):\Sigma^{m_4}\Sigma\Sigma\Sigma^{m_{[3,1]}}X_1\to\Sigma\Sigma^{m_4}\Sigma\Sigma^{m_{[3,1]}}X_1.
\end{align*}
Then 
\begin{equation}
\tau''_4\circ\tau'_4\circ\tau_4=1_{\Sigma^{m_{[3,1]}}X_1}\wedge\tau(\s^{m_4},\s^1)\wedge 1_{\s^1}.
\end{equation} 
Here we consider \fbox{Diagram $D_4$} in Appendix B.1. 

\begin{lemma'}
$\widetilde{\Sigma}^{m^*_3}a^*_{3,2}=
(1_{\Sigma^{m^*_3}C^*_{3,2}}\cup C\tau''_4)\circ\Phi(\widetilde{\Sigma}^{m_4}H^4)\circ(1_{\Sigma^{m_4}C_{4,3}}\cup C\tau_4)\circ\widetilde{\Sigma}^{m_4}a_{4,3}\circ\Sigma^{m_4}(1_{X_4}\cup C(\widetilde{\Sigma}^{m_3}a_{3,2})^{-1})
$.
\end{lemma'}
\begin{proof}
We have
\begin{align*}
&\widetilde{\Sigma}^{m^*_3}a^*_{3,2}=(1\cup C\tau''_4)\circ(\psi^{m^*_3}_{g^*_{3,2}})^{-1}\circ\Sigma^{m^*_3}a^*_{3,2}\\
&=(1\cup C\tau''_4)\circ(\psi^{m^*_3}_{g^*_{3,2}})^{-1}\circ\Sigma^{m_4}\Phi(H^4)\circ\Sigma^{m_4}a_{4,3}\circ\Sigma^{m_4}(1_{X_4}\cup C(\widetilde{\Sigma}^{m_3}a_{3,2})^{-1})\ (\text{by (3.1.3)})\\
&=(1\cup C\tau''_4)\circ\Phi(\widetilde{\Sigma}^{m_4}H^4)\circ(\psi^{m_4}_{g_{4,3}})^{-1}\circ\Sigma^{m_4}a_{4,3}\circ\Sigma^{m_4}(1_{X_4}\cup C(\widetilde{\Sigma}^{m_3}a_{3,2})^{-1})\\
&=(1\cup C\tau''_4)\circ\Phi(\widetilde{\Sigma}^{m_4}H^4)\circ(1\cup C\tau_4)\\
&\hspace{1cm}\circ(1\cup C\tau_4)^{-1}\circ(\psi^{m_4}_{g_{4,3}})^{-1}\circ\Sigma^{m_4}a_{4,3}\circ\Sigma^{m_4}(1_{X_4}\cup C(\widetilde{\Sigma}^{m_3}a_{3,2})^{-1})\\
&=(1\cup C\tau''_4)\circ\Phi(\widetilde{\Sigma}^{m_4}H^4)\circ(1\cup C\tau_4)
\circ\widetilde{\Sigma}^{m_4}a_{4,3}\circ \Sigma^{m_4}(1_{X_4}\cup C(\widetilde{\Sigma}^{m_3}a_{3,2})^{-1}) \ (\text{by (2.4)}).
\end{align*}
This ends the proof.
\end{proof}

We have
\begin{align*}
&\Sigma(\tau''_4\circ\tau'_4\circ\tau_4)\circ\widetilde{\Sigma}^{m_4}\omega_{4,3}\\
&=\Sigma\tau''_4\circ\Sigma\tau'_4\circ q'_{\Sigma^{m_4}g_{4,3}}\circ((1_{\Sigma^{m_4}C_{4,3}}\cup C\tau_4)\cup C1_{\Sigma^{m_4}C_{4,3}})\circ(\widetilde{\Sigma}^{m_4}a_{4,3}\cup C1_{\Sigma^{m_4}C_{4,3}})\\
&\simeq \Sigma\tau''_4\circ q'_{\Sigma^{m^*_3}g^*_{3,2}}\circ(\Phi(\widetilde{\Sigma}^{m_4}H^4)\cup C1_{\Sigma^{m_4}C_{4,3}})\circ ((1_{\Sigma^{m_4}C_{4,3}}\cup C\tau_4)\cup C1_{\Sigma^{m_4}C_{4,3}})\\
&\hspace{2cm}\circ(\widetilde{\Sigma}^{m_4}a_{4,3}\cup C1_{\Sigma^{m_4}C_{4,3}})\qquad (\text{by Lemma 2.1(2),(3)})\\
&=q'_{\widetilde{\Sigma}^{m^*_3}g^*_{3,2}}\circ
((1_{\Sigma^{m^*_3}C^*_{3,2}}\cup C\tau''_4)\cup C1_{\Sigma^{m^*_3}C^*_{3,2}})\circ(\Phi(\widetilde{\Sigma}^{m_4}H^4)\cup C1_{\Sigma^{m_4}C_{4,3}})\\
&\hspace{2cm}\circ((1_{\Sigma^{m_4}C_{4,3}}\cup C\tau_4)\cup C1_{\Sigma^{m_4}C_{4,3}})\circ(\widetilde{\Sigma}^{m_4}a_{4,3}\cup C1_{\Sigma^{m_4}C_{4,3}})\\
&\simeq
\widetilde{\Sigma}^{m^*_3}\omega^*_{3,2}\circ
((\widetilde{\Sigma}^{m^*_3}a^*_{3,2})^{-1}\cup C1_{\Sigma^{m^*_3}C^*_{3,2}})
\circ ((1_{\Sigma^{m^*_3}C^*_{3,2}}\cup C\tau''_4)\cup C1_{\Sigma^{m^*_3}C^*_{3,2}})\\
&\hspace{2cm}\circ(\Phi(\widetilde{\Sigma}^{m_4}H^4)\cup C1_{\Sigma^{m_4}C_{4,3}})
\circ((1_{\Sigma^{m_4}C_{4,3}}\cup C\tau_4)\cup C1_{\Sigma^{m_4}C_{4,3}})\\
&\hspace{2cm}\circ(\widetilde{\Sigma}^{m_4}a_{4,3}\cup C1_{\Sigma^{m_4}C_{4,3}})
\qquad(\text{by (2.7)})
\end{align*}
and so 
\begin{align*}
&(\widetilde{\Sigma}^{m_4}\omega_{4,3})^{-1}\circ\Sigma(\tau''_4\circ\tau'_4\circ\tau_4)^{-1}\\
&\simeq \big((\Sigma^{m_4}a_{4,3})^{-1}\circ(1\cup C\tau_4)^{-1}\circ\Phi(\widetilde{\Sigma}^{m_4}H^4)^{-1}\circ(1\cup C\tau''_4)^{-1}\circ
\widetilde{\Sigma}^{m^*_3}a^*_{3,2}\cup C1_{\Sigma^{m_4}C_{4,3}}\big)\\
&\hspace{4cm}\circ(\widetilde{\Sigma}^{m^*_3}\omega^*_{3,2})^{-1}\\
&\simeq(\Sigma^{m_4}(1_{X_4}\cup C(\widetilde{\Sigma}^{m_3}a_{3,2})^{-1})\cup C1_{\Sigma^{m_4}C_{4,3}})\circ(\widetilde{\Sigma}^{m^*_3}\omega^*_{3,2})^{-1}\ (\text{by Lemma 3.1.1}).
\end{align*}
Hence
$$
(\widetilde{\Sigma}^{m_4}\omega_{4,3})^{-1}\simeq\big(\Sigma^{m_4}(1_{X_4}\cup C(\widetilde{\Sigma}^{m_3}a_{3,2})^{-1})\cup C1_{\Sigma^{m_4}C_{4,3}}\big)\circ(\widetilde{\Sigma}^{m^*_3}\omega^*_{3,2})^{-1}\circ\Sigma(\tau''_4\circ\tau'_4\circ\tau_4)
$$
and so 
\begin{align*}
g_{5,4}&=(\overline{f_4}\cup C1_{\Sigma^{m_4}C_{4,3}})\circ(\widetilde{\Sigma}^{m_4}\omega_{4,3})^{-1}\\
&\simeq(\overline{f_4}\cup C1_{\Sigma^{m_4}C_{4,3}})\circ
\big(\Sigma^{m_4}(1_{X_4}\cup C(\widetilde{\Sigma}^{m_3}a_{3,2})^{-1})\cup C1_{\Sigma^{m_4}C_{4,3}}\big)\\
&\hspace{4cm}\circ(\widetilde{\Sigma}^{m^*_3}\omega^*_{3,2})^{-1}\circ\Sigma(\tau''_4\circ\tau'_4\circ\tau_4)\\
&=(\overline{f^*_3}\cup C1_{\Sigma^{m_4}C_{4,3}})\circ(\widetilde{\Sigma}^{m^*_3}\omega^*_{3,2})^{-1}\circ\Sigma(\tau''_4\circ\tau'_4\circ\tau_4)\\
&=g^*_{4,3}\circ(1_{\Sigma^{m_{[3,1]}}X_1}\wedge\tau(\s^{m_4},\s^1)\wedge1_{(\s^1)^{\wedge 2}})\quad(\text{by (3.1.4)})
\end{align*}
that is, 
$$
g_{5,4}\simeq g^*_{4,3}\circ(1_{\Sigma^{m_{[3,1]}}X_1}\wedge\tau(\s^{m_4},\s^1)\wedge1_{(\s^1)^{\wedge 2}}).
$$
By the last relation and (2.3), we easily have
\begin{equation}
\widetilde{\Sigma}^{m_5}g_{5,4}\simeq\widetilde{\Sigma}^{m^*_4}g^*_{4,3}\circ(1_{\Sigma^{m_{[3,1]}}X_1}\wedge\tau(\s^{m_{[5,4]}},\s^1)\wedge 1_{(\s^1)^{\wedge 2}}).
\end{equation}
Let $H^5:g_{5,4}\simeq  g^*_{4,3}\circ(1_{\Sigma^{m_{[3,1]}}X_1}\wedge\tau(\s^{m_4},\s^1)\wedge1_{(\s^1)^{\wedge 2}})$ be an any homotopy and set 
\begin{align}
\Phi(H^5)&=\Phi(g_{5,4},g^*_{4,3},1_{\Sigma^{m_{[3,1]}}X_1}\wedge\tau(\s^{m_4},\s^1)\wedge 1_{(\s^1)^{\wedge 2}},1_{C_{5,4}};H^5)\nonumber\\
&\hspace{2cm} : C_{5,4}\cup_{g_{5,4}}C\Sigma^2\Sigma\Sigma^{m_{[4,1]}}X_1\to C^*_{4,3}\cup_{g^*_{4,3}}C\Sigma^2\Sigma^{m^*_{[3,1]}}X^*_1\nonumber,\\
\begin{split}
a^*_{4,3}&=\Phi(H^5)\circ a_{5,4}\circ(1_{X_5}\cup C\Sigma^{m_4}(1_{X_4}\cup C(\widetilde{\Sigma}^{m_3}a_{3,2})^{-1}))\\
&\hspace{2cm}:C^*_{4,4}\to C^*_{4,3}\cup_{g^*_{4,3}}C\Sigma^2\Sigma^{m^*_{[3,1]}}X^*_1.
\end{split}
\end{align}
Let $\mathscr{A}^*_4$ be a reduced structure on $\mathscr{S}^*_4$ containing $a^*_{4,3}$ as a member. 

When $n=5$, $\{\mathscr{S}^*_r,\overline{f_4^*},\mathscr{A}^*_4\,|\,2\le r\le 4\}$ is an 
$\ddot{s}_t$-presentation of $(f^*_4,\dots,f^*_1)$ such that 
$$
\overline{f_5}\circ\widetilde{\Sigma}^{m_5}g_{5,4}\simeq \overline{f^*_4}\circ \widetilde{\Sigma}^{m^*_4}g^*_{4,3}\circ(1_{\Sigma^{m_{[3,1]}}X_1}\wedge\tau(\s^{m_{[5,4]}},\s^1)\wedge 1_{(\s^1)^{\wedge 2}})\quad(\text{by (3.1.5)}).
$$
Hence (3.1) holds for $n=5$. 

Fourthly let $n\ge 6$ and set $\mathscr{S}^*_5=(\widetilde{\Sigma}^{m^*_4}\mathscr{S}^*_4)(\overline{f^*_4},\widetilde{\Sigma}^{m_4^*}\mathscr{A}^*_4)$. 
Then 
$$
C^*_{5,s}=C_{6,s}\ (1\le s\le 3),\quad C^*_{5,4}=C_{6,5},\quad C^*_{5,5}=X^*_5\cup_{\overline{f^*_4}}C\Sigma^{m^*_4}C^*_{4,4}.
$$
Set
\begin{align*}
\overline{f^*_5}&=\begin{cases} \overline{f_6} & n=6\\ 
\overline{f_6}\circ\Sigma^{m^*_5}(1_{X_6}\cup C\Sigma^{m_4^*}(1_{X_5}\cup C\Sigma^{m^*_3}(1_{X_4}\cup C(\widetilde{\Sigma}^{m_3}a_{3,2})^{-1}))) & n\ge 7
\end{cases},\\
\tau_5&=1_{\Sigma^{m_{[4,1]}}X_1}\wedge\tau(\s^{m_5},(\s^1)^{\wedge 3}):\Sigma^3\Sigma^{m_{[5,1]}}X_1\to \Sigma^{m_5}\Sigma^3\Sigma^{m_{[4,1]}}X_1,\\
\tau'_5&=1_{\Sigma^{m_{[3,1]}}X_1}\wedge\tau(\s^{m_4},\s^1)\wedge 1_{(\s^1)^{\wedge 2}}\wedge 1_{\s^{m_5}}:\Sigma^{m_5}\Sigma^3\Sigma^{m_{[4,1]}}X_1\to \Sigma^{m^*_4}\Sigma^2\Sigma^{m^*_{[3,1]}}X^*_1,\\
\tau''_5&=1_{\Sigma^{m^*_3}X^*_1}\wedge\tau((\s^1)^{\wedge 2},\s^{m^*_4}):\Sigma^{m^*_4}\Sigma^2\Sigma^{m^*_{[3,1]}}X^*_1\to\Sigma^2\Sigma^{m^*_{[4,1]}}X^*_1.
\end{align*}
Then 
\begin{equation}
\tau''_5\circ\tau'_5\circ\tau_5=1_{\Sigma^{m_{[3,1]}}X_1}\wedge\tau(\s^{m_{[5,4]}},\s^1)\wedge 1_{(\s^1)^{\wedge 2}}.
\end{equation}
Also we set
\begin{align*}
&\Phi(\widetilde{\Sigma}^{m_5}H^5)=\Phi(\Sigma^{m_5}g_{5,4},\Sigma^{m^*_4}g^*_{4,3},\tau'_5,1_{\Sigma^{m_5}C_{5,4}};\widetilde{\Sigma}^{m_5}H^5)\\
&\hspace{1cm}:\Sigma^{m_5}C_{5,4}\cup_{\Sigma^{m_5}g_{5,4}}C\Sigma^{m_5}\Sigma^3\Sigma^{m_{[4,1]}}X_1\to\Sigma^{m^*_4}C^*_{4,3}\cup_{\Sigma^{m^*_4}g^*_{4,3}}C\Sigma^{m^*_4}\Sigma^2\Sigma^{m^*_{[3,1]}}X^*_1,\\
&\Phi'(\widetilde{\Sigma}^{m_5}H^5)=\Phi(i_{\Sigma^{m_5}g_{5,4}},i_{\Sigma^{m^*_4}g^*_{4,3}},1_{\Sigma^{m_5}C_{5,4}},\Phi(\widetilde{\Sigma}^{m_5}H^5);
1_{i_{\Sigma^{m^*_4}g^*_{4,3}}})\\
&\hspace{1cm}:\Sigma^{m_5}C_{5,4}\cup_{\Sigma^{m_5}g_{5,4}}C\Sigma^{m_5}\Sigma^3\Sigma^{m_{[4,1]}}X_1\to \Sigma^{m^*_4}C^*_{4,3}\cup_{\widetilde{\Sigma}^{m^*_4}g^*_{4,3}}C\Sigma^{m^*_4}\Sigma^2\Sigma^{m^*_{[3,1]}}X^*_1.
\end{align*}
Here we consider \fbox{Diagram $D_5$} in Appendix B.1. 

\begin{lemma'}
We have 
\begin{align*}
\widetilde{\Sigma}^{m^*_4}a^*_{4,3}
&=(1_{\Sigma^{m^*_4}C^*_{4,3}}\cup C\tau''_5)\circ\Phi(\widetilde{\Sigma}^{m_5}H^5)
\circ(1_{\Sigma^{m_5}C_{5,4}}\cup C\tau_5)\circ\widetilde{\Sigma}^{m_5}a_{5,4}\\
&\hspace{4cm}\circ\Sigma^{m_5}(1_{X_5}\cup C\Sigma^{m_4}(1_{X_4}\cup C(\widetilde{\Sigma}^{m_3}a_{3,2})^{-1})).
\end{align*}
\end{lemma'}
\begin{proof}
We have
\begin{align*}
&\widetilde{\Sigma}^{m^*_4}a^*_{4,3}=(1_{\Sigma^{m^*_4}C^*_{4,3}}\cup C\tau''_5)\circ(\psi^{m^*_4}_{g^*_{4,3}})^{-1}\circ\Sigma^{m^*_4}a^*_{4,3}\\
&=(1_{\Sigma^{m^*_4}C^*_{4,3}}\cup C\tau''_5)\circ(\psi^{m^*_4}_{g^*_{4,3}})^{-1}\circ
\Sigma^{m^*_4}\Phi(H^5)\\
&\hspace{1cm}\circ\Sigma^{m^*_4}a_{5,4}
\circ\Sigma^{m^*_4}(1_{X_5}\cup C\Sigma^{m_4}(1_{X_4}\cup C(\widetilde{\Sigma}^{m_3}a_{3,2})^{-1}))\quad(\text{by (3.1.6)})\\
&=(1_{\Sigma^{m^*_4}C^*_{4,3}}\cup C\tau''_5)\circ\Phi(\widetilde{\Sigma}^{m_5}H^5)\circ(\psi^{m_5}_{g_{5,4}})^{-1}\\
&\hspace{1cm}\circ\Sigma^{m^*_4}a_{5,4}
\circ\Sigma^{m^*_4}(1_{X_5}\cup C\Sigma^{m_4}(1_{X_4}\cup C(\widetilde{\Sigma}^{m_3}a_{3,2})^{-1}))\\
&=(1_{\Sigma^{m^*_4}C^*_{4,3}}\cup C\tau''_5)\circ\Phi(\widetilde{\Sigma}^{m_5}H^5)\circ(1_{\Sigma^{m_5}C_{5,4}}\cup C\tau_5)\circ\\
&\hspace{1cm}(1_{\Sigma^{m_5}C_{5,4}}\cup C\tau_5)^{-1}\circ (\psi^{m_5}_{g_{5,4}})^{-1}\circ
\Sigma^{m^*_4}a_{5,4}
\circ\Sigma^{m^*_4}(1_{X_5}\cup C\Sigma^{m_4}(1_{X_4}\cup C(\widetilde{\Sigma}^{m_3}a_{3,2})^{-1}))\\
&=(1_{\Sigma^{m^*_4}C^*_{4,3}}\cup C\tau''_5)\circ\Phi(\widetilde{\Sigma}^{m_5}H^5)\circ(1_{\Sigma^{m_5}C_{5,4}}\cup C\tau_5)\circ\\
&\hspace{1cm}\widetilde{\Sigma}^{m_5}a_{5,4}\circ\Sigma^{m^*_4}(1_{X_5}\cup C\Sigma^{m_4}(1_{X_4}\cup C(\widetilde{\Sigma}^{m_3}a_{3,2})^{-1}))\quad(\text{by (2.4)}).
\end{align*}
This ends the proof.
\end{proof}

We have
\begin{align*}
&\Sigma(\tau''_5\circ\tau'_5\circ\tau_5)\circ\widetilde{\Sigma}^{m_5}\omega_{5,4}
=\Sigma\tau''_5\circ\Sigma\tau'_5\circ\Sigma\tau_5\circ q'_{\widetilde{\Sigma}^{m_5}g_{5,4}}\circ(\widetilde{\Sigma}^{m_5}a_{5,4}\cup C1_{\Sigma^{m_5}C_{5,4}})\\
&=\Sigma\tau''_5\circ\Sigma\tau'_5\circ q'_{\Sigma^{m_5}g_{5,4}}\circ((1_{\Sigma^{m_5}C_{5,4}})\cup C\tau_5)\cup C1_{\Sigma^{m_5}C_{5,4}})\circ(\widetilde{\Sigma}^{m_5}a_{5,4}\cup C1_{\Sigma^{m_5}C_{5,4}})\\
&\simeq\Sigma\tau''_5\circ q'_{\Sigma^{m^*_4}g^*_{4,3}}\circ\Phi'(\widetilde{\Sigma}^{m_5}H^5)\\
&\hspace{1cm}\circ((1_{\Sigma^{m_5}C_{5,4}})\cup C\tau_5)\cup C1_{\Sigma^{m_5}C_{5,4}})\circ(\widetilde{\Sigma}^{m_5}a_{5,4}\cup C1_{\Sigma^{m_5}C_{5,4}})\ (\text{by Lemma 2.1(1)})\\
&=q'_{\widetilde{\Sigma}^{m^*_4}g^*_{4,3}}\circ((1_{\Sigma^{m^*_4}g^*_{4,3}}\cup 
C\tau''_5)\cup C1_{\Sigma^{m^*_4}C^*_{4,3}})\circ(\Phi(\widetilde{\Sigma}^{m_5}H^5)\cup C1_{\Sigma^{m_5}C_{5,4}})\\
&\hspace{1cm}\circ((1_{\Sigma^{m_5}C_{5,4}})\cup C\tau_5)\cup C1_{\Sigma^{m_5}C_{5,4}})\circ(\widetilde{\Sigma}^{m_5}a_{5,4}\cup C1_{\Sigma^{m_5}C_{5,4}})\ (\text{by Lemma 2.1(2),(3)})\\
&=q'_{\widetilde{\Sigma}^{m^*_4}g^*_{4,3}}\circ\big((1_{\Sigma^{m^*_4}g^*_{4,3}}\cup 
C\tau''_5)\circ\Phi(\widetilde{\Sigma}^{m_5}H^5)\circ(1_{\Sigma^{m_5}C_{5,4}}\cup C\tau_5)\circ\widetilde{\Sigma}^{m_5}a_{5,4}\cup C1_{\Sigma^{m_5}C_{5,4}}\big)\\
&\simeq \widetilde{\Sigma}^{m^*_4}\omega^*_{4,3}\circ((\widetilde{\Sigma}^{m^*_4}a^*_{4,3})^{-1}\cup C1_{\Sigma^{m^*_4}C^*_{4,3}})\\
&\hspace{5mm}\circ \big((1_{\Sigma^{m^*_4}g^*_{4,3}}\cup 
C\tau''_5)\circ\Phi(\widetilde{\Sigma}^{m_5}H^5)\circ(1_{\Sigma^{m_5}C_{5,4}}\cup C\tau_5)\circ\widetilde{\Sigma}^{m_5}a_{5,4}\cup C1_{\Sigma^{m_5}C_{5,4}}\big)\ (\text{by (2.7)})\\
&\simeq\widetilde{\Sigma}^{m^*_4}\omega^*_{4,3}\circ\Big(\big(\Sigma^{m_5}(1_{X_5}\cup C\Sigma^{m_4}(1_{X_4}\cup C(\widetilde{\Sigma}^{m_3}a_{3,2})^{-1}))\big)^{-1}\cup C1_{\Sigma^{m_5}C_{5,4}}\Big)\\
&\hspace{6cm} (\text{by Lemma 3.1.2}).
\end{align*}
Hence
\begin{align*}
(\widetilde{\Sigma}^{m_5}\omega_{5,4})^{-1}&\simeq(\Sigma^{m_5}(1_{X_5}\cup C\Sigma^{m_4}(1_{X_4}\cup C(\widetilde{\Sigma}^{m_3}a_{3,2})^{-1}))\cup C1_{\Sigma^{m_5}C_{5,4}})\\
&\hspace{1cm}\circ(\widetilde{\Sigma}^{m^*_4}\omega^*_{4,3})^{-1}\circ 
(1_{\Sigma^{m_{[3,1]}}X_1}\wedge\tau(\s^{m_{[5,4]}},\s^1)\wedge 1_{(\s^1)^{\wedge 3}})
\end{align*}
by (3.1.7), and so 
\begin{align*}
g_{6,5}&=(\overline{f_5}\cup C1_{\Sigma^{m_5}C_{5,4}})\circ(\widetilde{\Sigma}^{m_5}\omega_{5,4})^{-1}\\
&\simeq(\overline{f_5}\cup C1_{\Sigma^{m_5}C_{5,4}})\circ(\Sigma^{m_5}(1_{X_5}\cup C\Sigma^{m_4}(1_{X_4}\cup C(\widetilde{\Sigma}^{m_3}a_{3,2})^{-1}))\cup C1_{\Sigma^{m_5}C_{5,4}})\\
&\hspace{1cm}\circ(\widetilde{\Sigma}^{m^*_4}\omega^*_{4,3})^{-1}\circ(1_{\Sigma^{m_{[3,1]}}X_1}\wedge\tau(\s^{m_{[5,4]}},\s^1)\wedge 1_{(\s^1)^{\wedge 3}})
\\
&=(\overline{f^*_4}\cup C1_{\Sigma^{m^*_4}C^*_{4,3}})\circ(\widetilde{\Sigma}^{m^*_4}\omega^*_{4,3})^{-1}\circ(1_{\Sigma^{m_{[3,1]}}X_1}\wedge\tau(\s^{m_{[5,4]}},\s^1)\wedge 1_{(\s^1)^{\wedge 3}})
\\
&=g^*_{5,4}\circ(1_{\Sigma^{m_{[3,1]}}X_1}\wedge\tau(\s^{m_{[5,4]}},\s^1)\wedge 1_{(\s^1)^{\wedge 3}}),
\end{align*}
that is,
$$
g_{6,5}\simeq g^*_{5,4}\circ(1_{\Sigma^{m_{[3,1]}}X_1}\wedge\tau(\s^{m_{[5,4]}},\s^1)\wedge 1_{(\s^1)^{\wedge 3}}).
$$
By the last relation, we easily have
\begin{equation}
\widetilde{\Sigma}^{m_6}g_{6,5}\simeq\widetilde{\Sigma}^{m^*_5}g^*_{5,4}\circ(1_{\Sigma^{m_{[3,1]}}X_1}\wedge\tau(\s^{m_{[6,4]}},\s^1)\wedge 1_{(\s^1)^{\wedge 3}}).
\end{equation}
Let $H^6:g_{6,5}\simeq g^*_{5,4}\circ(1_{\Sigma^{m_{[3,1]}}X_1}\wedge\tau(\s^{m_{[5,4]}},\s^1)\wedge 1_{(\s^1)^{\wedge 3}})$ be an any homotopy and set 
\begin{align*}
\Phi(H^6)&=\Phi(g_{6,5},g^*_{5,4},1_{\Sigma^{m_{[3,1]}}X_1}\wedge\tau(\s^{m_{[5,4]}},\s^1)\wedge 1_{(\s^1)^{\wedge 3}},1_{C_{6,5}};H^6)\\
&\hspace{2cm}:C_{6,5}\cup_{g_{6,5}}C\Sigma^4\Sigma^{m_{[5,1]}}X_1\to C^*_{5,4}\cup_{g^*_{5,4}}C\Sigma^3\Sigma^{m^*_{[4,1]}}X^*_1,\\
 a^*_{5,4}&=\Phi(H^6)\circ a_{6,5}\circ\big(1_{X_6}\cup C\Sigma^{m_5}(1_{X_5}\cup C\Sigma^{m_4}(1_{X_4}\cup C(\widetilde{\Sigma}^{m_3}a_{3,2})^{-1}))\big)\\
&\hspace{2cm}:C^*_{5,5}\to C^*_{5,4}\cup_{g^*_{5,4}}C\Sigma^3\Sigma^{m^*_{[4,1]}}X^*_1.
\end{align*}
Let $\mathscr{A}^*_5$ be a reduced structure on $\mathscr{S}^*_5$ containing $a^*_{5,4}$ as a member. 

When $n=6$, $\{\mathscr{S}^*_r,\overline{f^*_r},\mathscr{A}^*_r\,|\,2\le r\le 5\}$ is an $\ddot{s}_t$-presentation of $(f^*_5,\dots,f^*_1)$ such that 
$$
\overline{f_6}\circ\widetilde{\Sigma}^{m_6}g_{6,5}\simeq \overline{f^*_5}\circ\widetilde{\Sigma}^{m^*_5}g^*_{5,4}\circ(1_{\Sigma^{m_{[3,1]}}X_1}\wedge\tau(\s^{m_{[5,4]}},\s^1)\wedge 1_{(\s^1)^{\wedge 3}})\quad(\text{by (3.1.8)}).
$$
Thus (3.1) holds for $n=6$. 

Fifthly let $n\ge 7$ and set $\mathscr{S}^*_6=(\widetilde{\Sigma}^{m^*_5}\mathscr{S}^*_5)(\overline{f^*_5},\widetilde{\Sigma}^{m^*_5}\mathscr{A}^*_5)$. 
Then 
$$
C^*_{6,s}=C_{7,s}\ (1\le s\le 4),\quad C^*_{6,5}=C_{7,6},\quad C^*_{6,6}=X^*_6\cup_{\overline{f^*_5}}CC^*_{5,5}.
$$
Set
$$
\overline{f_6^*}=\begin{cases}\overline{f_7} & n=7\\
\overline{f_7}\circ\Sigma^{m_6^*}(1_{X_7}\cup C\Sigma^{m^*_5}(\cdots\cup C\Sigma^{m^*_3}(1_{X_4}\cup C(\widetilde{\Sigma}^{m_3}a_{3,2})^{-1})\cdots)) & n\ge 8\end{cases}.
$$
Proceeding with the above construction, 
we obtain inductively a desired $\ddot{s}_t$-presentation of $\vec{\bm f^*}$. 
This ends the proof of (3.1). 

\subsection{Proof of (3.2)}
Let $n\ge 4$. 
It suffices to prove (3.2) when the term on the right hand is not empty. 
Suppose that 
$$
\{f_n,\dots,f_4,[f_3,A_2,\Sigma^{m_3}f_2]\circ(\psi^{m_3}_{f_2})^{-1},\psi^{m_3}_{f_2}\circ(\Sigma^{m_3}f_2,\widetilde{\Sigma}^{m_3}A_1,\Sigma^{m_{[3,2]}}f_1)\}^{(\ddot{s}_t)}_{(m_n,\dots,m_4,0,0)}
$$
is not empty for 
$A_2:f_3\circ \Sigma^{m_3}f_2\simeq *$ and $A_1:f_2\circ \Sigma^{m_2}f_1\simeq *$. 
We set 
\begin{align*}
&f^*_1=\psi^{m_3}_{f_2}\circ(\Sigma^{m_3}f_2,\widetilde{\Sigma}^{m_3}A_1,\Sigma^{m_{[3,2]}}f_1):X^*_1\to X^*_2,\\
& f^*_2=[f_3,A_2,\Sigma^{m_3}f_2]\circ(\psi^{m_3}_{f_2})^{-1}:X^*_2\to X^*_3,\\
& f^*_k=f_{k+1}:\Sigma^{m^*_k}X^*_k\to X^*_{k+1}\ (3\le k< n).
\end{align*}
Note that $f^*_1=\Sigma^{m_3}(f_2,A_1,\Sigma^{m_2}f_1)\circ(1_{\Sigma^{m_{[2,1]}}X_1}\wedge\tau(\s^{m_3},\s^1))$ by \cite[Lemma 2.4]{OO1}. 
Take $\alpha\in\{f^*_{n-1},\dots,f^*_1\}^{(\ddot{s}_t)}_{(m^*_{n-1},\dots,m^*_1)}$ and let 
$\{\mathscr{S}^*_r,\overline{f^*_r},\mathscr{A}^*_r\,|\,2\le r< n\}$ be an $\ddot{s}_t$-presentation of $\overrightarrow{\bm f^*}$ such that $\alpha=\overline{f^*_{n-1}}\circ\widetilde{\Sigma}^{m^*_{n-1}}g^*_{n-1,n-2}$. 
We will construct an $\ddot{s}_t$-presentation $\{\mathscr{S}_r,\overline{f_r},\mathscr{A}_r\,|\,2\le r\le n\}$ of $\vec{\bm f}$ such that 
$$
\overline{f_n}\circ\widetilde{\Sigma}^{m_n}g_{n,n-1}\simeq \overline{f^*_{n-1}}\circ\widetilde{\Sigma}^{m^*_{n-1}}g^*_{n-1,n-2}\circ(1_{\Sigma^{m_{[3,1]}}X_1}\wedge\tau(\s^{m_{[n,4]}},\s^1)\wedge 1_{(\s^1)^{\wedge(n-3)}}).
$$
Once we have such an $\ddot{s}_t$-presentation, then we obtain (3.2). 
Set 
\begin{gather*}
\mathscr{S}_2=(\Sigma^{m_1}X_1;X_2,X_2\cup_{f_1}C\Sigma^{m_1}X_1;f_1,i_{f_1}),\ \mathscr{A}_2=\{1_{C_{2,2}}\},\\
\overline{f_2}=[f_2,A_1,\Sigma^{m_2}f_1]\circ(\psi^{m_2}_{f_1})^{-1}:\Sigma^{m_2}C_{2,2}\to X_3,\\
\mathscr{S}_3=(\widetilde{\Sigma}^{m_2}\mathscr{S}_2)(\overline{f_2},\widetilde{\Sigma}^{m_2}\mathscr{A}_2)\ \text{with}\ g_{3,2}=(f_2,A_1,\Sigma^{m_2}f_1). 
\end{gather*}
Then $f^*_1=\widetilde{\Sigma}^{m_3}g_{3,2}$. 
Take a homotopy equivalence $a_{3,2}\in\mathrm{TOP}^{C_{3,2}}(j_{3,2},i_{g_{3,2}})$. 
Set $\mathscr{A}_3=\{1_{C_{3,2}}, a_{3,2}\}$ which is a reduced structure on $\mathscr{S}_3$. 
By definition 
$$
\widetilde{\Sigma}^{m_3}a_{3,2}= (1_{\Sigma^{m_3}C_{3,2}}\cup C(1_{\Sigma^{m_{[2,1]}}X_1}\wedge\tau(\s^1,\s^{m_3})))\circ(\psi^{m_3}_{g_{3,2}})^{-1}\circ\Sigma^{m_3}a_{3,2}:\Sigma^{m_3}C_{3,3}\to C^*_{2,2}.
$$
By (2.5), we have
\begin{align*}
\widetilde{\Sigma}^{m_3}\omega_{3,2}&=q'_{\widetilde{\Sigma}^{m_3}g_{3,2}}\circ(\widetilde{\Sigma}^{m_3}a_{3,2}\cup C1_{\Sigma^{m_3}C_{3,2}})=\omega^*_{2,1}
\circ(\widetilde{\Sigma}^{m_3}a_{3,2}\cup C1_{\Sigma^{m_3}C_{3,2}})\\
&:\Sigma^{m_3}C_{3,3}\cup C\Sigma^{m_3}C_{3,2}\to \Sigma\Sigma\Sigma^{m_{[3,1]}}X_1
\end{align*}
and so 
$$
(\omega^*_{2,1})^{-1}\simeq (\widetilde{\Sigma}^{m_3}a_{3,2}\cup C1_{\Sigma^{m_3}C_{3,2}})\circ(\widetilde{\Sigma}^{m_3}\omega_{3,2})^{-1}.
$$
Set 
$$
\overline{f_3}=\overline{f^*_2}\circ\widetilde{\Sigma}^{m_3}a_{3,2}:\Sigma^{m_3}C_{3,3}\to X_4
$$
which is an extension of $f_3$. 
Set 
\begin{align*}
&\mathscr{S}_4=(\widetilde{\Sigma}^{m_3}\mathscr{S}_3)(\overline{f_3},\widetilde{\Sigma}^{m_3}\mathscr{A}_3),\\
&\overline{f_4}=\begin{cases} \overline{f^*_3} :\Sigma^{m_4}C_{4,3}=\Sigma^{m^*_3}C^*_{3,2}\to X^*_4=X_5 & n=4\\
\overline{f^*_3}\circ\Sigma^{m_4}(1_{X_4}\cup C\widetilde{\Sigma}^{m_3}a_{3,2}):\Sigma^{m_4}C_{4,4}\to X_5 & n\ge 5\end{cases}.
\end{align*}
Then
$$
C_{4,1}=C^*_{3,1},\quad C_{4,3}=C^*_{3,2},\quad C_{4,4}=X_4\cup_{\overline{f_3}}C\Sigma^{m_3}C_{3,3}.
$$
We have
\begin{align*}
g^*_{3,2}&=(\overline{f_2^*}\cup C1_{X^*_2})\circ(\omega^*_{2,1})^{-1}\\
&\simeq(\overline{f_2^*}\cup C1_{X^*_2})\circ(\widetilde{\Sigma}^{m_3}a_{3,2}\cup C1_{\Sigma^{m_3}C_{3,2}})\circ(\widetilde{\Sigma}^{m_3}\omega_{3,2})^{-1}\\
&=(\overline{f_3}\cup C1_{\Sigma^{m_3}C_{3,2}})\circ(\widetilde{\Sigma}^{m_3}\omega_{3,2})^{-1}
=g_{4,3}.
\end{align*}
Let $H^4:g^*_{3,2}\simeq g_{4,3}$ be an any homotopy and set
\begin{align*}
&\Phi(H^4)=\Phi(g^*_{3,2},g_{4,3},1_{\Sigma^2\Sigma^{m_{[3,1]}}X_1},1_{C^*_{3,2}};H^4)\\
&\hspace{3cm}:C^*_{3,2}\cup_{g^*_{3,2}}C\Sigma\Sigma^{m^*_{[2,1]}}X^*_1\to C_{4,3}\cup_{g_{4,3}}C\Sigma^2\Sigma^{m_{[3,1]}}X_1,\\
&a_{4,3}=\Phi(H^4)\circ a^*_{3,2}\circ(1_{X_4}\cup C\widetilde{\Sigma}^{m_3}a_{3,2}) : 
C_{4,4}\to C_{4,3}\cup_{g_{4,3}}C\Sigma^2\Sigma^{m_{[3,1]}}X_1.
\end{align*}
Let $\mathscr{A}_4$ be a reduced structure on $\mathscr{S}_4$ containing $a_{4,3}$ as a member. 

When $n=4$, $\{\mathscr{S}_r,\overline{f_r},\mathscr{A}_r\,|\,2\le r\le 4\}$ is an $\ddot{s}_t$-presentation of $(f_4,\dots,f_1)$ such that 
\begin{align*}
&\overline{f_4}\circ\widetilde{\Sigma}^{m_4}g_{4,3}=\overline{f_4}\circ\Sigma^{m_4}g_{4,3}\circ(1_{\Sigma^{m_{[3,1]}}X_1}\wedge\tau(\s^{m_4},(\s^1)^2))\\
&\simeq \overline{f_4}\circ\Sigma^{m_3^*}g^*_{3,2}\circ(1_{\Sigma^{m_{[3,1]}}X_1}\wedge\tau(\s^{m_4},(\s^1)^2))\\
&=\overline{f^*_3}\circ\widetilde{\Sigma}^{m_3^*}g^*_{3,2}\circ (1_{\Sigma^{m_{[3,1]}}X_1}\wedge\tau(\s^{m_4},\s^1)\wedge 1_{\s^1})
\end{align*}
so that (3.2) holds for $n=4$. 

Let $n\ge 5$. 
We have
\begin{align*}
&\omega_{4,3}=q'_{g_{4,3}}\circ(a_{4,3}\cup C1_{C_{4,3}})=q'_{g_{4,3}}\circ\big(\Phi(H^4)\circ a^*_{3,2}\circ(1_{X_4}\cup C\widetilde{\Sigma}^{m_3}a_{3,2})\cup C1_{C_{4,3}}\big)\\
&=q'_{g_{4,3}}\circ(\Phi(H^4)\cup C1_{C_{4,3}})\circ(a^*_{3,2}\cup C1_{C_{4,3}})\circ\big((1_{X_4}\cup C\widetilde{\Sigma}^{m_3}a_{3,2})\cup C1_{C_{4,3}}\big)\\
&\simeq q'_{g^*_{3,2}}\circ(a^*_{3,2}\cup C1_{C_{4,3}})\circ\big((1_{X_4}\cup C\widetilde{\Sigma}^{m_3}a_{3,2})\cup C1_{C_{4,3}}\big)\ (\text{by Lemma 2.1(3)})\\
&=\omega^*_{3,2}\circ\big((1_{X_4}\cup C\widetilde{\Sigma}^{m_3}a_{3,2})\cup C1_{C_{4,3}}\big)
\end{align*}
and so
\begin{align*}
&\widetilde{\Sigma}^{m_4}\omega_{4,3}=(1_{\Sigma^{m_{[3,1]}}X_1}\wedge\tau((\s^1)^{\wedge 3},\s^{m_4}))\circ\Sigma^{m_4}\omega_{4,3}\circ\psi^{m_4}_{j_{4,3}}\\
&\simeq (1_{\Sigma^{m_{[3,1]}}X_1}\wedge\tau((\s^1)^{\wedge 3},\s^{m_4}))\circ\Sigma^{m_4}\omega^*_{3,2}\circ\Sigma^{m_4}\big((1_{X_4}\cup C\widetilde{\Sigma}^{m_3}a_{3,2})\cup C1_{C_{4,3}}\big)\circ\psi^{m_4}_{j_{4,3}}\\
&=(1_{\Sigma^{m_{[3,1]}}X_1}\wedge\tau((\s^1)^{\wedge 3},\s^{m_4}))\circ\Sigma^{m_4}\omega^*_{3,2}\circ\psi^{m_4}_{j^*_{3,2}}\circ\big(\Sigma^{m_4}(1_{X_4}\cup C\widetilde{\Sigma}^{m_3}a_{3,2})\cup C1_{C_{4,3}}\big)\\
&=(1_{\Sigma^{m_{[3,1]}}X_1}\wedge\tau((\s^1)^{\wedge 3},\s^{m_4}))\circ(1_{\Sigma\Sigma^{m_{[3,1]}}X_1}\wedge\tau(\s^{m^*_3},(\s^1)^{\wedge 2}))\circ\widetilde{\Sigma}^{m^*_3}\omega^*_{3,2}\\
&\hspace{2cm}\circ\big(\Sigma^{m_4}(1_{X_4}\cup C\widetilde{\Sigma}^{m_3}a_{3,2})\cup C1_{C_{4,3}}\big)\\
&=(1_{\Sigma^{m_{[3,1]}}X_1}\wedge\tau(\s^1,\s^{m_4})\wedge 1_{(\s^1)^{\wedge 2}})
\circ\widetilde{\Sigma}^{m^*_3}\omega^*_{3,2}\circ(\Sigma^{m_4}(1_{X_4}\cup C\widetilde{\Sigma}^{m_3}a_{3,2})\cup C1_{C_{4,3}}\big).
\end{align*}
Hence
\begin{equation}
\begin{split}
(\widetilde{\Sigma}^{m_4}\omega_{4,3})^{-1}\simeq\big(\Sigma^{m_4}(1_{X_4}&\cup C\widetilde{\Sigma}^{m_3}a_{3,2})\cup C1_{\Sigma^{m_4}C_{4,3}}\big)^{-1}\circ(\widetilde{\Sigma}^{m^*_3}\omega^*_{3,2} )^{-1}\\
&\circ(1_{\Sigma^{m_{[3,1]}}X_1}\wedge\tau(\s^{m_4},\s^1)\wedge 1_{(\s^1)^{\wedge 2}}).
\end{split}
\end{equation}
Set $\mathscr{S}_5=(\widetilde{\Sigma}^{m_4}\mathscr{S}_4)(\overline{f_4},\widetilde{\Sigma}^{m_4}\mathscr{A}_4)$. 
Then 
$$
C_{5,s}=C^*_{4,s}\ (1\le s\le 2), \quad C_{5,4}=C^*_{4,3},\quad  C_{5,5}=X_5\cup_{\overline{f_4}}C\Sigma^{m_4}C_{4,4}.
$$
By (3.2.1), we have
\begin{align*}
&g_{5,4}=(\overline{f_4}\cup C1_{\Sigma^{m_4}C_{4,3}})\circ(\widetilde{\Sigma}^{m_4}\omega_{4,3})^{-1}\\
&\simeq (\overline{f_4}\cup C1_{\Sigma^{m_4}C_{4,3}})\circ(\Sigma^{m_4}(1_{X_4}\cup C\widetilde{\Sigma}^{m_3}a_{3,2})\cup C1_{\Sigma^{m_4}C_{4,3}})^{-1}\circ(\widetilde{\Sigma}^{m^*_3}\omega^*_{3,2})^{-1}\\
&\hspace{2cm}\circ(1_{\Sigma^{m_{[3,1]}}X_1}\wedge\tau(\s^{m_4},\s^1)\wedge 1_{(\s^1)^{\wedge 2}})\\
&\simeq(\overline{f^*_3}\cup C1_{\Sigma^{m_4}C_{4,3}})\circ(\widetilde{\Sigma}^{m^*_3}\omega^*_{3,2})^{-1}\circ(1_{\Sigma^{m_{[3,1]}}X_1}\wedge\tau(\s^{m_4},\s^1)\wedge 1_{(\s^1)^{\wedge 2}})\\
&=g^*_{4,3}\circ((1_{\Sigma^{m_{[3,1]}}X_1}\wedge\tau(\s^{m_4},\s^1)\wedge 1_{(\s^1)^{\wedge 2}}),
\end{align*}
that is, 
\begin{equation}
g^*_{4,3}\simeq g_{5,4}\circ((1_{\Sigma^{m_{[3,1]}}X_1}\wedge\tau(\s^1,\s^{m_4})\wedge 1_{(\s^1)^{\wedge 2}}).
\end{equation}
Let $H^5:g^*_{4,3}\simeq g_{5,4}\circ((1_{\Sigma^{m_{[3,1]}}X_1}\wedge\tau(\s^1,\s^{m_4})\wedge 1_{(\s^1)^{\wedge 2}})$ be an any homotopy and set 
$$
a_{5,4}=\Phi(H^5)\circ a^*_{4,3}\circ(1_{X_5}\cup C\Sigma^{m_4}(1_{X_4}\cup C\widetilde{\Sigma}^{m_3}a_{3,2})):C_{5,5}\to C_{5,4}\cup_{g_{5,4}}C\Sigma^3\Sigma^{m_{[4,1]}}X_1.
$$
Then $a_{5,4}\in\mathrm{TOP}^{C_{5,4}}(j_{5,4},i_{g_{5,4}})$ is a homotopy equivalence in the category $\mathrm{TOP}^{C_{5,4}}$. 
Let $\mathscr{A}_5$ be a reduced structure on $\mathscr{S}_5$ containing $a_{5,4}$ as a member. 
By (2.3) and (3.2.2), we have
\begin{equation}
\widetilde{\Sigma}^{m_5}g_{5,4}\simeq \widetilde{\Sigma}^{m^*_4}g^*_{4,3}\circ(1_{\Sigma^{m_{[3,1]}}X_1}\wedge\tau(\s^{m_{[5,4]}},\s^1)\wedge 1_{(\s^1)^{\wedge 2}}).
\end{equation}
Set
$$
\overline{f_5}=\begin{cases}\overline{f^*_4}:\Sigma^{m_5}C_{5,4}\to X_6 & n=5\\
\overline{f^*_4}\circ\Sigma^{m_5}(1_{X_5}\cup C\Sigma^{m_4}(1_{X_4}\cup C\widetilde{\Sigma}^{m_3}a_{3,2})):\Sigma^{m_5}C_{5,5}\to X_6 & n\ge 6\end{cases}
$$
which is an extension of $f_5$. 

When $n=5$, $\{\mathscr{S}_r,\overline{f_r},\mathscr{A}_r\,|\,2\le r\le 5\}$ is an $\ddot{s}_t$-presentation of $(f_5,\dots,f_1)$ such that 
\begin{align*}
&\overline{f_5}\circ\widetilde{\Sigma}^{m_5}g_{5,4}= \overline{f^*_4}\circ 
\widetilde{\Sigma}^{m_5}g_{5,4}\\
&\simeq \overline{f^*_4}\circ \widetilde{\Sigma}^{m^*_4}g^*_{4,3}\circ(1_{\Sigma^{m_{[3,1]}}X_1}\wedge\tau(\s^{m_{[5,4]}},\s^1)\wedge 1_{(\s^1)^{\wedge 2}})\quad(\text{by (3.2.3)}).
\end{align*}
Hence (3.2) holds for $n=5$. 

Let $n\ge 6$. 
Set $\mathscr{S}_6=(\widetilde{\Sigma}^{m_5}\mathscr{S}_5)(\overline{f_5},\widetilde{\Sigma}^{m_5}\mathscr{A}_5)$. 
Then 
$$
C_{6,s}=C^*_{5,s}\ (1\le s\le 3), \quad C_{6,5}=C^*_{5,4}, \quad  C_{6,6}=X_6\cup_{\overline{f_5}}C\Sigma^{m_5}C_{5,5}.
$$ 
Set
\begin{gather*}
\overline{f_6}=\begin{cases} \overline{f^*_5}:\Sigma^{m_6}C_{6,5}=\Sigma^{m^*_5}C^*_{5,4}\to X_7 & n=6\\
\overline{f^*_5}\circ\Sigma^{m_6}(1_{X_6}\cup C\Sigma^{m_5}(1_{X_5}\cup C\Sigma^{m_4}(1_{X_4}\cup C\widetilde{\Sigma}^{m_3}a_{3,2}))):\Sigma^{m_6}C_{6,6}\to X_7 & n\ge 7\end{cases},\\
\tau_5=1_{\Sigma^{m_4}\Sigma\Sigma^{m_{[3,1]}}X_1}\wedge\tau(\s^{m_5},\s^2),\\
\tau'_5=1_{\Sigma^{m_{[3,1]}}X_1}\wedge\tau(\s^1,\s^{m_4})\wedge 1_{\s^2}\wedge 1_{\s^{m_5}},\\
\tau''_5=1_{\Sigma^{m_{[4,1]}}X_1}\wedge\tau(\s^3,\s^{m_5}),\\
\Phi(\widetilde{\Sigma}^{m_5}H^5)=\Phi(\Sigma^{m^*_4}g^*_{4,3},\Sigma^{m_5}g_{5,4},\tau'_5,1_{\Sigma^{m_5}C_{5,4}};\widetilde{\Sigma}^{m_5}H^5),\\
\Phi'(\widetilde{\Sigma}^{m_5}H^5)=\Phi(i_{\Sigma^{m^*_4}g^*_{4,3}},i_{\Sigma^{m_5}g_{5,4}},1_{\Sigma^{m_5}C_{5,4}},\Phi(\widetilde{\Sigma}^{m_5}H^5);1_{i_{\Sigma^{m_5}g_{5,4}}}).\end{gather*}
Note that 
$\tau''_5\circ\tau'_5\circ\tau_5=1_{\Sigma^{m_{[3,1]}}X_1}\wedge\tau(\s^1,\s^{m_{[5,4]}})\wedge 1_{\s^1}$. 
Here we consider \fbox{Diagram $D^*_5$} in Appendix B.2. 

\begin{lemma'}
$\widetilde{\Sigma}^{m_5}a_{5,4}=(1\cup C\tau''_5)\circ\Phi(\widetilde{\Sigma}^{m_5}H^5)\circ(1\cup C\tau_5)\circ\widetilde{\Sigma}^{m^*_4}a^*_{4,3}\circ\Sigma^{m_5}(1_{X_5}\cup C\Sigma^{m_4}(1_{X_4}\cup C\widetilde{\Sigma}^{m_3}a_{3,2}))$.
\end{lemma'}
\begin{proof}
We have
\begin{align*}
&\widetilde{\Sigma}^{m_5}a_{5,4}=(1_{\Sigma^{m_5}C_{5,4}}\cup C\tau''_5)\circ(\psi^{m_5}_{g_{5,4}})^{-1}\circ\Sigma^{m_5}a_{5,4}\\
&=(1_{\Sigma^{m_5}C_{5,4}}\cup C\tau''_5)\circ(\psi^{m_5}_{g_{5,4}})^{-1}\circ\Sigma^{m_5}\Phi(H^5)\circ\Sigma^{m_5}a^*_{4,3}\\
&\hspace{4cm}\circ\Sigma^{m_5}(1_{X_5}\cup C\Sigma^{m_4}(1_{X_4}\cup C\widetilde{\Sigma}^{m_3}a_{3,2}))\\
&=(1_{\Sigma^{m_5}C_{5,4}}\cup C\tau''_5)\circ\Phi(\widetilde{\Sigma}^{m_5}H^5)\circ(\psi^{m_5}_{g^*_{4,3}})^{-1}\circ\Sigma^{m_5}a^*_{4,3}\\
&\hspace{4cm}\circ\Sigma^{m_5}(1_{X_5}\cup C\Sigma^{m_4}(1_{X_4}\cup C\widetilde{\Sigma}^{m_3}a_{3,2}))\\
&=(1_{}\cup C\tau''_5)\circ\Phi(\widetilde{\Sigma}^{m_5}H^5)\circ(1\cup C\tau_5)
\circ(1_{}\cup C\tau_5)^{-1}\circ(\psi^{m_5}_{g^*_{4,3}})^{-1}\circ\Sigma^{m_5}a^*_{4,3}\\
&\hspace{4cm}\circ\Sigma^{m_5}(1_{X_5}\cup C\Sigma^{m_4}(1_{X_4}\cup C\widetilde{\Sigma}^{m_3}a_{3,2}))\\
&=(1_{}\cup C\tau''_5)\circ\Phi(\widetilde{\Sigma}^{m_5}H^5)\circ(1\cup C\tau_5)\circ\widetilde{\Sigma}^{m_5}a^*_{4,3}\circ\Sigma^{m_5}(1_{X_5}\cup C\Sigma^{m_4}(1_{X_4}\cup C\widetilde{\Sigma}^{m_3}a_{3,2})).
\end{align*}
This ends the proof.
\end{proof}
It follows from Lemma 2.1 that $\Phi'(\widetilde{\Sigma}^{m_5}H^5)\simeq \Phi(\widetilde{\Sigma}^{m_5}H^5)\cup C1_{\Sigma^{m_5}C_{5,4}}$ and 
$\Sigma\tau'_5\circ q'_{\Sigma^{m_4^*}g^*_{4,3}}\simeq q'_{\Sigma^{m_5}g_{5,4}}\circ\Phi'(\widetilde{\Sigma}^{m_5}H^5)$ so that 
$$
\Sigma\tau'_5\circ q'_{\Sigma^{m_4^*}g^*_{4,3}}\simeq q'_{\Sigma^{m_5}g_{5,4}}\circ(\Phi(\widetilde{\Sigma}^{m_5}H^5)\cup C1_{\Sigma^{m_5}C_{5,4}}). 
$$
We then have
\begin{align*}
&\Sigma(\tau''_5\circ\tau'_5\circ\tau_5)\circ\widetilde{\Sigma}^{m^*_4}\omega^*_{4,3}\\
&\simeq q'_{\widetilde{\Sigma}^{m_5}g_{5,4}}\circ((1\cup C\tau''_5)\cup C1)\circ(\Phi(\widetilde{\Sigma}^{m_5}H^5)\cup C1)\circ((1\cup C\tau_5)\cup C1)\circ(\widetilde{\Sigma}^{m^*_4}a^*_{4,3}\cup C1)\\
&\simeq \widetilde{\Sigma}^{m_5}\omega_{5,4}\circ((\widetilde{\Sigma}^{m_5}a_{5,4})^{-1}\cup C1)\circ((1\cup C\tau''_5)\cup C1)\circ(\Phi(\widetilde{\Sigma}^{m_5}H^5)\cup C1)\\
&\hspace{2cm}\circ((1\cup C\tau_5)\cup C1)\circ(\widetilde{\Sigma}^{m^*_4}a^*_{4,3}\cup C1)\\
&\simeq \widetilde{\Sigma}^{m_5}\omega_{5,4}\circ\big(\Sigma^{m_5}(1_{X_5}\cup C\Sigma^{m_4}(1_{X_4}\cup C\widetilde{\Sigma}^{m_3}a_{3,2}))^{-1}\cup C1_{\Sigma^{m_5}C_{5,4}}\big)\quad(\text{by Lemma 3.2.1}).
\end{align*}
Hence
$$
\Sigma(\tau''_5\circ\tau'_5\circ\tau_5)\circ\widetilde{\Sigma}^{m^*_4}\omega^*_{4,3}
\circ\big(\Sigma^{m_5}(1_{X_5}\cup C\Sigma^{m_4}(1_{X_4}\cup C\widetilde{\Sigma}^{m_3}a_{3,2}))\cup C1_{\Sigma^{m^*_4}C^*_{4,3}}\big)\simeq\widetilde{\Sigma}^{m_5}\omega_{5,4}
$$
and
\begin{align*}
g_{6,5}&=(\overline{f_5}\cup C1_{\Sigma^{m_5}C_{5,4}})\circ(\widetilde{\Sigma}^{m_5}\omega_{5,4})^{-1}\\
&\simeq (\overline{f_5}\cup C1_{\Sigma^{m_5}C_{5,4}})\circ(\Sigma^{m_5}(1_{X_5}\cup C\Sigma^{m_4}(1_{X_4}\cup C\widetilde{\Sigma}^{m_3}a_{3,2}))^{-1}\cup C1_{\Sigma^{m_4^*}C^*_{4,3}})\\
&\hspace{4cm}\circ(\widetilde{\Sigma}^{m^*_4}\omega^*_{4,3})^{-1}
\circ\Sigma(\tau''_5\circ\tau'_5\circ\tau_5)^{-1}\\
&\simeq(\overline{f^*_4}\cup C1_{\Sigma^{m^*_4}C^*_{4,3}})\circ(\widetilde{\Sigma}^{m^*_4}\omega^*_{4,3})^{-1}\circ\Sigma(\tau''_5\circ\tau'_5\circ\tau_5)^{-1}\\
&=g^*_{5,4}\circ(1_{\Sigma^{m_{[3,1]}}X_1}\wedge\tau(\s^{m_{[5,4]}},\s^1)\wedge 1_{(\s^1)^{\wedge 2}})
\end{align*}
Let $H^6:g^*_{5,4}\simeq g_{6,5}\circ (1_{\Sigma^{m_{[3,1]}}X_1}\wedge\tau(\s^1,\s^{m_{[5,4]}})\wedge 1_{(\s^1)^{\wedge 2}})$ be an any homotopy and set
\begin{align*}
a_{6,5}&=\Phi(H^6)\circ a^*_{5,4}\circ(1_{X_6}\cup C\Sigma^{m_5}(1_{X_5}\cup C\Sigma^{m_4}(1_{X_4}\cup C\widetilde{\Sigma}^{m_3}a_{3,2})))\\
&:C_{6,6}\to C_{6,5}\cup_{g_{6,5}}C\Sigma^4\Sigma^{m_{[5,1]}}X_1.
\end{align*}
Let $\mathscr{A}_6$ be a reduced structure on $\mathscr{S}_6$ containing $a_{6,5}$ as a member. 

When $n=6$, $\{\mathscr{S}_r,\overline{f_r},\mathscr{A}_r\,|\,2\le r\le 6\}$ is an $\ddot{s}_t$-presentation of $(f_6,\dots,f_1)$ such that 
\begin{align*}
&\overline{f_6}\circ\widetilde{\Sigma}^{m_6}g_{6,5}=\overline{f^*_5}\circ\Sigma^{m_6}g_{6,5}\circ(1_{\Sigma_{m_{[5,1]}}X_1}\wedge\tau(\s^{m_6},\s^4))\\
&\simeq\overline{f^*_5}\circ\Sigma^{m^*_5}g^*_{5,4}\circ(1_{\Sigma^{m_{[3,1]}}X_1}\wedge\tau(\s^{m_{[5,4]}},\s^1)\wedge 1_{(\s^1)^{\wedge 3}}\wedge 1_{\s^{m_6}})\circ(1_{\Sigma^{m_{[5,1]}}X_1}\wedge\tau(\s^{m_6},\s^4))\\
&=\overline{f^*_5}\circ\widetilde{\Sigma}^{m^*_5}g^*_{5,4}\circ(1_{\Sigma^{m^*_{[4,1]}}X^*_1}\wedge\tau(\s^3,\s^{m^*_5}))\circ (1_{\Sigma^{m_{[3,1]}}X_1}\wedge\tau(\s^{m_{[5,4]}},\s^1)\wedge 1_{(\s^1)^{\wedge 3}}\wedge 1_{\s^{m_6}})\\
&\hspace{4cm}\circ(1_{\Sigma^{m_{[5,1]}}X_1}\wedge\tau(\s^{m_6},\s^4))\\
&=\overline{f^*_5}\circ\widetilde{\Sigma}^{m^*_5}g^*_{5,4}\circ(1_{\Sigma^{m_{[3,1]}}X_1}\wedge\tau(\s^{m_{[6,4]}},\s^1)\wedge 1_{(\s^1)^{\wedge 3}}).
\end{align*}
Hence (3.2) holds for $n=6$.

Let $n\ge 7$. 
Proceeding the above process, we obtain inductively an $\ddot{s}_t$-presentation $\{\mathscr{S}_r,\overline{f_r},\mathscr{A}_r\,|\,2\le r\le n\}$ of $(f_n,\dots,f_1)$ such that 
$$
\overline{f_n}\circ\widetilde{\Sigma}^{m_n}g_{n,n-1}\simeq \overline{f^*_{n-1}}\circ\widetilde{\Sigma}^{m^*_{n-1}}g^*_{n-1,n-2}\circ(1_{\Sigma^{m_{[3,1]}}X_1}\wedge\tau(\s^{m_{[n,4]}},\s^1)\wedge 1_{(\s^1)^{\wedge(n-3)}}).
$$
This ends the proof of (3.2). 

\section{Proof of Corollary 1.2}
We use mathematical induction on $n$. 
The first case $n=3$ holds by \cite[Proposition~2.3]{T2}. 
Suppose that the assertion holds for some $n\ge 3$. 
We will prove the assertion for $n+1$. 
Suppose that $X_{n+2}=\s^{m+1}$, $X_k$ is a connected pointed CW complex for all $k$ with $n+1\ge k\ge 1$, $\vec{\bm m}=(m_{n+1},\dots,m_1)$ is a sequence of non negative integers with $m_{n+1}\ge 1$, and $\vec{\bm f}=(f_{n+1},\dots,f_1)$ is a sequence of 
pointed maps $f_k:\Sigma^{m_k}X_k\to X_{k+1}\ (1\le k\le n+1)$. 
We must prove $H\{\vec{\bm f}\,\}^{(\ddot{s}_t)}_{\vec{\bm m}}\subset\{H(f_{n+1}),f_n,\dots,f_1\,\}^{(\ddot{s}_t)}_{\vec{\bm m}}$. 
Set 
\begin{align*}
\varphi&=1_{\Sigma^{m_{[3,1]}}X_1}\wedge\tau(\s^{m_{[n+1,4]}},\s^1)\wedge 1_{\s^{n-2}}\\
&:\Sigma^{n-1}\Sigma^{m_{[n+1,1]}}X_1=\Sigma^{n-2}\Sigma\Sigma^{m_{[n+1,1]}}X_1\to \Sigma^{n-2}\Sigma^{m_{[n+1,4]}}\Sigma\Sigma^{m_{[3,1]}}X_1.
\end{align*}
We have 
\begin{align*}
&H\{\vec{\bm f}\,\}^{(\ddot{s}_t)}_{\vec{\bm m}}\\
&=\bigcup_{A_2,A_1}H\circ\varphi^*\{f_{n+1},\dots,f_4,[f_3,A_2,\Sigma^{m_3}f_2],(\Sigma^{m_3}f_2,\widetilde{\Sigma}^{m_3}A_1,\Sigma^{m_{[3,2]}}f_1)\}^{(\ddot{s}_t)}_{(m_{n+1},\dots,m_4,0,0)}\\
&\hspace{6cm}(\text{by Theorem 1.1})\\
&=\bigcup_{A_2,A_1}\varphi^*\circ H\{f_{n+1},\dots,f_4,[f_3,A_2,\Sigma^{m_3}f_2],(\Sigma^{m_3}f_2,\widetilde{\Sigma}^{m_3}A_1,\Sigma^{m_{[3,2]}}f_1)\}^{(\ddot{s}_t)}_{(m_{n+1},\dots,m_4,0,0)}\\
&\hspace{6cm}(\text{since $H\circ\varphi^*=\varphi^*\circ H$})\\
&\subset\bigcup_{A_2,A_1}\varphi^*\{H(f_{n+1}),f_n,\dots,f_4,[f_3,A_2,\Sigma^{m_3}f_2],(\Sigma^{m_3}f_2,\widetilde{\Sigma}^{m_3}A_1,\Sigma^{m_{[3,2]}}f_1)\}^{(\ddot{s}_t)}_{(m_{n+1},\dots,m_4,0,0)}\\
&\hspace{6cm}(\text{by the inductive assumption})\\
&=\{H(f_{n+1}),f_n,\dots,f_1\}^{(\ddot{s}_t)}_{\vec{\bm m}}\qquad(\text{by Theorem 1.1}).
\end{align*}
This completes the induction and ends the proof of Corollary 1.2.

\section{Brackets in $\mathrm{TOP}^*$}
In this section, we work in $\mathrm{TOP}^*$. 

Let $n\ge 3$ be an integer, $(X_{n+1},\dots,X_1)$ a sequence of spaces, 
$\vec{\bm m}=(m_n,\dots,m_1)$ a sequence of non-negative integers, and $\vec{\bm f}=(f_n,\dots,f_1)$ a sequence of maps $f_k:\Sigma^{m_k}X_k\to X_{k+1}$. 
We will define three brackets $\{\vec{\bm f}\,\}_{\vec{\bm m}}$, $\{\vec{\bm f}\,\}_{\vec{\bm m}}'$, and $\{\vec{\bm f}\,\}_{\vec{\bm m}}''$. 

For $n=3$, they are the classical Toda bracket $\{f_3,f_2,\Sigma^{m_2}f_1\}_{m_3}$, that is, 
we define 
\begin{align*}
&\{\vec{\bm f}\,\}_{\vec{\bm m}}=\{\vec{\bm f}\,\}_{\vec{\bm m}}'=\{\vec{\bm f}\,\}_{\vec{\bm m}}''\\
&=\bigcup_{A_2,A_1}\{[f_3,A_2,\Sigma^{m_3}f_2],(\Sigma^{m_3}f_2,\widetilde{\Sigma}^{m_3}A_1,\Sigma^{m_{[3,2]}}f_1)\}_{(0,0)}
\subset[\Sigma\Sigma^{m_{[3,1]}}X_1,X_4].
\end{align*} 

For $n\ge 4$, we define $\{\vec{\bm f}\}_{\vec{\bm m}}$, $\{\vec{\bm f}\}_{\vec{\bm m}}'$, and $\{\vec{\bm f}\}_{\vec{\bm m}}''$ inductively by 
\begin{gather*}
\{\vec{\bm f}\}_{\vec{\bm m}}=\bigcup_{\text{admissible\ }\vec{\bm A}}\{f_n,\dots,f_4,[f_3,A_2,\Sigma^{m_3}f_2],(\Sigma^{m_3}f_2,\widetilde{\Sigma}^{m_3}A_1,\Sigma^{m_{[3,2]}}f_1)\}_{(m_n,\dots,m_4,0,0)},\\
\{\vec{\bm f}\}_{\vec{\bm m}}'=\bigcup_{\text{all\ }\vec{\bm A}}\{f_n,\dots,f_4,[f_3,A_2,\Sigma^{m_3}f_2],(\Sigma^{m_3}f_2,\widetilde{\Sigma}^{m_3}A_1,\Sigma^{m_{[3,2]}}f_1)\}_{(m_n,\dots,m_4,0,0)}',\\
\{\vec{\bm f}\}_{\vec{\bm m}}''=\bigcup_{A_2,A_1}\{f_n,\dots,f_4,[f_3,A_2,\Sigma^{m_3}f_2],(\Sigma^{m_3}f_2,\widetilde{\Sigma}^{m_3}A_1,\Sigma^{m_{[3,2]}}f_1)\}_{(m_n,\dots,m_4,0,0)}'',
\end{gather*}
where unions $\bigcup_{\text{admissible}\,\vec{\bm A}}$\,, $\bigcup_{\text{all}\,\vec{\bm A}}$\,, and $\bigcup_{A_2,A_1}$ are ones used in Theorem 1.1. 
By mathematical induction, we easily have

\begin{lemma}
If $n\ge 4$, then $\{\vec{\bm f}\}_{\vec{\bm m}}\subset\{\vec{\bm f}\}_{\vec{\bm m}}'\subset\{\vec{\bm f}\}_{\vec{\bm m}}''\subset[(\Sigma\Sigma^{m_n})\cdots(\Sigma\Sigma^{m_4})\Sigma\Sigma^{m_{[3,1]}}X_1,X_{n+1}]$.
\end{lemma}

Define a homeomorphism
$$
h_{\vec{\bm m}}:(\s^{m_{[3,1]}}\wedge\s^1)\wedge(\s^{m_4}\wedge\s^1)\wedge\cdots\wedge(\s^{m_n}\wedge\s^1)\underset{\approx}{\to}\s^{m_{[n,1]}}\wedge\underbrace{\s^1\wedge\cdots\wedge\s^1}_{n-2}=\s^{m_{[n,1]}}\wedge\s^{n-2}
$$
by 
$h_{\vec{\bm m}}(s_1\wedge s_2\wedge s_3\wedge\overline{t_3}\wedge(s_4\wedge\overline{t_4})\wedge\cdots\wedge(s_n\wedge\overline{t_n}))
=s_1\wedge\cdots\wedge s_n\wedge\overline{t_3}\wedge\cdots\wedge\overline{t_n}
$ for $s_k\in\s^{m_k},\,\overline{t_k}\in\s^1$. 
By mathematical induction, it follows from Theorem 1.1 that 

\begin{prop}
If $n\ge 4$ and $X_k$ is well pointed for all $k$ with $1\le k\le n+1$, then 
$\{\vec{\bm f}\,\}_{\vec{\bm m}}=\{\vec{\bm f}\,\}_{\vec{\bm m}}'=\{\vec{\bm f}\,\}''_{\vec{\bm m}}=\{\vec{\bm f}\}_{\vec{\bm m}}^{(\ddot{s}_t)}\circ(1_{X_1}\wedge h_{\vec{\bm m}})$.
\end{prop}

\begin{prob}
Does $\{\vec{\bm f}\,\}_{\vec{\bm m}}=\{\vec{\bm f}\,\}_{\vec{\bm m}}'=\{\vec{\bm f}\,\}_{\vec{\bm m}}''$ hold always for $n\ge 4$?
\end{prob}

If pointed versions of propositions 2.1 and 2.2 of \cite{OO2} hold, that is, if the conjecture below is affirmative, then our consideration in \cite{OO2,OO3} can be developed in $\mathrm{TOP}^*$ and Problem 5.3 above is affirmative. 

\begin{conj}
\begin{enumerate}
\item Given a commutative triangle in $\mathrm{TOP}^*$
$$
\xymatrix{
& A \ar[dl]_-i \ar[dr]^-{i'}  &\\
X  \ar[rr]^-f  & &X'
}
$$
if $i$ and $i'$ are pointed cofibrations and $f$ is a pointed homotopy equivalence, then $f:i\to i'$ is a homotopy equivalence in $\mathrm{TOP}^{*A}$, the category of pointed spaces under the pointed space $A$.
\item If a pointed map $j:A\to X$ is a pointed cofibration, then for any pointed map $f:X\to Y$ 
the pointed map $1_Y\cup Cj:Y\cup_{f\circ j}CA\to Y\cup_f CX$ is a pointed cofibration.
\end{enumerate}
\end{conj}

\begin{rem}
We can prove that $\{\vec{\bm f}\,\}_{\vec{\bm m}}$ depends only on homotopy classes of $f_k\ (1\le k\le n)$ 
and that if $n\ge 3$ and $\{\vec{\bm f}\,\}_{\vec{\bm m}}$ is not empty then $f_{k+1}\circ\Sigma^{m_{k+1}}f_k\simeq *$ 
for all $k$ with $n>k\ge 1$ and $\{f_k,\dots,f_2,f_1\}_{(m_k,\dots,m_1)}$ contains $0$ for all $k$ with $n>k\ge 2$, 
where $\{f_2,f_1\}_{(m_2,m_1)}$ denotes the one point set consisting of the homotopy class of $f_2\circ\Sigma^{m_2}f_1$. 
Similar statements hold for $\{\vec{\bm f}\,\}_{\vec{\bm m}}'$ and $\{\vec{\bm f}\,\}_{\vec{\bm m}}''$. 
\end{rem}


\begin{appendix}
\section{Proof of Lemma 3.1} 
The following lemma includes Lemma 3.1.
 
\begin{lemma}
Let $(f_n,\dots,f_0)$ be a sequence of maps $f_k:\Sigma^{m_k}X_k\to X_{k+1}\ (0\le k\le n)$ with $n\ge 3$. 
\begin{enumerate}
\item We have
$$
\{f_n,\dots,f_2,f_1\circ\Sigma^{m_1}f_0\}^{(\ddot{s}_t)}_{(m_n,\dots,m_2,m_1+m_0)}
\subset\{f_n,\dots,f_3,f_2\circ\Sigma^{m_2}f_1,f_0\}^{(\ddot{s}_t)}_{(m_n,\dots,m_3,m_2+m_1,m_0)}.
$$
\item If $m_1=0$ and $f_1:X_1\to X_2$ is a homotopy equivalence, then 
$$
\{f_n,\dots,f_2,f_1\circ f_0\}^{(\ddot{s}_t)}_{(m_n,\dots,m_2,m_0)}=\{f_n,\dots,f_3,f_2\circ\Sigma^{m_2}f_1,f_0\}^{(\ddot{s}_t)}_{(m_n,\dots,m_2,m_0)}.
$$
\end{enumerate}
\end{lemma}
\begin{proof}
We use the following notations: 
\begin{gather*}
(X''_{n+1},\dots,X''_1)=(X_{n+1},\dots,X_2,X_0),\quad (m_n'',\dots,m_1'')=(m_n,\dots,m_2,m_1+m_0),\\
(f''_n,\dots,f''_1)=(f_n,\dots,f_2,f_1\circ\Sigma^{m_1}f_0),\\
(X^*_{n+1},\dots,X^*_1)=(X_{n+1},\dots,X_3,X_1,X_0),\quad (m_n^*,\dots,m_1^*)=(m_n,\dots,m_3,m_2+m_1,m_0),\\
(f^*_n,\dots,f^*_1)=(f_n,\dots,f_3,f_2\circ\Sigma^{m_2}f_1,f_0).
\end{gather*}

(1) We must show that $\{\vec{\bm f''}\}^{(\ddot{s}_t)}_{\vec{\bm m''}}\subset\{\vec{\bm f^*}\}^{(\ddot{s}_t)}_{\vec{\bm m^*}}$. 
Let $\alpha\in\{\vec{\bm f''}\}^{(\ddot{s}_t)}_{\vec{\bm m''}}$ and 
$\{\mathscr{S}_r'',\overline{f_r''},\mathscr{A}_r''\,|\,2\le r\le n\}$ an 
$\ddot{s}_t$-presentation of $\vec{\bm f''}$ with 
$\alpha=\overline{f''_n}\circ \widetilde{\Sigma}^{m_n}g''_{n,n-1}$. 
Set 
\begin{gather*}
\mathscr{S}_2^*=(\Sigma^{m^*_1}X^*_1;X^*_2,X^*_2\cup_{f^*_1}C\Sigma^{m^*_1}X^*_1;f^*_1;i_{f^*_1})=(\Sigma^{m_0}X_0;X_1, X_1\cup_{f_0}C\Sigma^{m_0}X_0;f_0;i_{f_0}),\\
e_{2,2}=(f_1\cup C1_{\Sigma^{m_{[1,0]}}X_0})\circ (\psi^{m_1}_{f_0})^{-1}:\Sigma^{m_1}C^*_{2,2}\to C''_{2,2},\\
\overline{f^*_2}=\overline{f_2''}\circ \Sigma^{m_2}e_{2,2}:\Sigma^{m^*_2}C^*_{2,2}\to X_3''=X_3^*=X_3,\quad 
\mathscr{A}^*_2=\{1_{C^*_{2,2}}\},\\
\mathscr{S}_3^*=(\widetilde{\Sigma}^{m_2^*}\mathscr{S}^*_2)(\overline{f^*_2},\widetilde{\Sigma}^{m_2^*}\mathscr{A}^*_2),\quad 
e_{3,1}=1_{X_3}:C^*_{3,1}=X_3^*\to X_3''=C''_{3,1},\\
 e_{3,2}=1_{X_3}\cup C\Sigma^{m_2}f_1:C^*_{3,2}=X_3\cup_{f_2\circ \Sigma^{m_2}f_1}C\Sigma^{m_{[2,1]}}X_1\to C''_{3,2}=X_3\cup_{f_2}C\Sigma^{m_2}X_2.
\end{gather*}
By definition the following diagram is commutative.
$$
\xymatrix{
C^*_{3,2} \ar[d]^-{e_{3,2}} & \Sigma^{m^*_2}C^*_{2,2}\cup C\Sigma^{m^*_2}X^*_2 \ar[d]^-{\Sigma^{m_2}e_{2,2}\cup C\Sigma^{m_2}f_1} \ar[l]_-{\overline{f^*_2}\cup C1} \ar[r]^-{\widetilde{\Sigma}^{m^*_2}\omega^*_{2,1}} & \Sigma\Sigma^{m_{[2,0]}}X_0 \ar@{=}[d]\\
C''_{3,2} & \Sigma^{m''_2}C''_{2,2}\cup C\Sigma^{m''_2}X''_2 \ar[l]_-{\overline{f''_2}\cup C1} \ar[r]^-{\widetilde{\Sigma}^{m''_2}\omega''_{2,1}} & \Sigma\Sigma^{m_{[2,0]}}X_0
}
$$
Hence $e_{3,2}\circ g^*_{3,2}\simeq g''_{3,2}$. 
Let $\mathscr{A}_3^*$ be an arbitrary reduced structure on $\mathscr{S}^*_3$. 

Let $n=3$. 
Set $\overline{f^*_3}=\overline{f''_3}\circ\Sigma^{m_3}e_{3,2}:\Sigma^{m_3}e_{3,2}:\Sigma^{m_3}C^*_{3,2}\to X^*_4=X_4$. 
We can easily see that $\{\mathscr{S}^*_r,\overline{f^*_r},\mathscr{A}^*_r\,|\,r=2,3\}$ 
is an $\ddot{s}_t$-presentation of $\vec{\bm f^*}$ and 
$\overline{f^*_3}\circ\widetilde{\Sigma}^{m_3^*}g^*_{3,2}\simeq \overline{f''_3}\circ \widetilde{\Sigma}^{m''_3}g''_{3,2}$. 
Hence $\alpha\in\{f_3^*,f_2^*,f_1^*\}^{(\ddot{s}_t)}_{(m^*_3,m_2^*,m_1^*)}$. 
This proves (1) for $n=3$. 

If $n\ge 4$, let $J^3:e_{3,2}\circ g^*_{3,2}\simeq g''_{3,2}$ be an any homotopy and set 
$\Phi(J^3)=\Phi(g^*_{3,2},g''_{3,2},1_{},e_{3,2};J^3):C_{g^*_{3,2}}\to C_{g''_{3,2}}$. 
By Lemma 2.1, we have the following homotopy commutative diagram. 
$$
\xymatrix{
&& C^*_{3,3} \ar[d]^-{a^*_{3,2}} \ar[r] & C^*_{3,3}\cup CC^*_{3,2} \ar[d]_-{a^*_{3,2}\cup C1} \ar[dr]^-{\omega^*_{3,2}} & \\
\Sigma\Sigma^{m_{[2,0]}}X_0 \ar@{=}[d] \ar[r]^-{g^*_{3,2}} & C^*_{3,2} \ar[d]_-{e_{3,2}} \ar[r] \ar[ur]^-{j^*_{3,2}} & C_{g^*_{3,2}} \ar[d]_-{\Phi(J^3)} \ar[r] & C_{g^*_{3,2}}\cup CC^*_{3,2} \ar[d]^-{\Phi(J^3)\cup Ce_{3,2}} \ar[r] & \Sigma^2 \Sigma^{m_{[2,0]}}X_0 \ar@{=}[d] \\
\Sigma\Sigma^{m_{[2,0]}}X_0 \ar[r]^-{g''_{3,2}} & C''_{3,2} \ar[r] \ar[dr]_-{j''_{3,2}} & C_{g''_{3,2}} \ar[d]^-{{a''_{3,2}}^{-1}} \ar[r] & C_{g''_{3,2}}\cup CC''_{3,2} \ar[d]_-{{a''_{3,2}}^{-1}\cup C1} \ar[r] & \Sigma^2 \Sigma^{m_{[2,0]}}X_0\\
&& C''_{3,3} \ar[r] & C''_{3,3}\cup CC''_{3,2} \ar[ur]_-{\omega''_{3,2}} &
}
$$
Set 
\begin{gather*}
e_{3,3}={a''_{3,2}}^{-1}\circ\Phi(J^3)\circ a^*_{3,2}:C^*_{3,3}\to C''_{3,3},\\ 
\overline{f^*_3}=\overline{f''_3}\circ \Sigma^{m_3}e_{3,3}:\Sigma^{m_3}C^*_{3,3}\to X_4,\\ 
\mathscr{S}^*_4=(\widetilde{\Sigma}^{m_3}\mathscr{S}^*_3)(\overline{f^*_3},\widetilde{\Sigma}^{m_3}\mathscr{A}^*_3),\\
e_{4,3}=1_{X_4}\cup C\Sigma^{m_3}e_{3,2}:C^*_{4,3}\to C''_{4,3}. 
\end{gather*}
Then the following diagram is homotopy commutative. 
$$
\xymatrix{
C^*_{4,3} \ar[d]^-{e_{4,3}} & \Sigma^{m_3}C^*_{3,3}\cup C\Sigma^{m_3}C^*_{3,2} \ar[d]^-{\Sigma^{m_3}e_{3,3}\cup C\Sigma^{m_3}e_{3,2}} \ar[l]_-{\overline{f^*_3}\cup C1} \ar[r]^-{\widetilde{\Sigma}^{m^*_3}\omega^*_{3,2}} & \Sigma^2\Sigma^{m_{[3,0]}}X_0 \ar@{=}[d]\\
C''_{4,3} & \Sigma^{m_3}C''_{3,3}\cup C\Sigma^{m_3}C''_{3,2} \ar[l]_-{\overline{f''_3}\cup C1} \ar[r]^-{\widetilde{\Sigma}^{m_3}\omega''_{3,2}} & \Sigma^2\Sigma^{m_{[3,0]}}X_0
}
$$
Hence $e_{4,3}\circ g^*_{4,3}\simeq g''_{4,3}$. 
Let $\mathscr{A}_4^*$ be an arbitrary reduced structure on $\mathscr{S}^*_4$. 

Let $n=4$. 
Set $\overline{f^*_4}=\overline{f''_4}\circ \Sigma^{m_4}e_{4,3}:\Sigma^{m_4}C^*_{4,3}\to X_5$. 
Then $\{\mathscr{S}^*_r,\overline{f^*_r},\mathscr{A}^*_r\,|\,r=2,3,4\}$ is an $\ddot{s}_t$-presentation of $\vec{\bm f^*}$ and
$\overline{f^*_4}\circ \widetilde{\Sigma}^{m_4}g^*_{4,3}\simeq \overline{f''_4}\circ\widetilde{\Sigma}^{m_4}g''_{4,3}$. 
This proves (1) for $n=4$. 

For $n\ge 5$, let $J^4:e_{4,3}\circ g^*_{4,3}\simeq g''_{4,3}$ be an any homotopy and set 
\begin{gather*}
e_{4,4}={a''_{4,3}}^{-1}\circ\Phi(g^*_{4,3},g''_{4,3},1,e_{4,3};J^4)\circ a^*_{4,3}:C^*_{4,4}\to C''_{4,4},\\
\overline{f^*_4}=\overline{f''_4}\circ\Sigma^{m_4}e_{4,4}:\Sigma^{m_4}C^*_{4,4}\to X_5,\\
\mathscr{S}^*_5=(\widetilde{\Sigma}^{m_4}\mathscr{S}^*_4)(\overline{f^*_4},\widetilde{\Sigma}^{m_4}\mathscr{A}^*_4),\\
e_{5,4}=1_{X_5}\cup C\Sigma^{m_4}e_{4,3}:C^*_{5,4}\to C''_{5,4}.
\end{gather*} 
Let $\mathscr{A}^*_5$ be an arbitrary reduced structure on $\mathscr{S}^*_5$. 
Proceeding with the process above, we have an $\ddot{s}_t$-presentation of $\vec{\bm f^*}$ with $\overline{f^*_n}\circ\widetilde{\Sigma}^{m_n}g^*_{n,n-1}\simeq \overline{f''_n}\circ\widetilde{\Sigma}^{m_n}g''_{n,n-1}$ so that (1) 
is proved for all $n\ge 5$. 
This completes the proof of (1). 

(2) Suppose that $m_1=0$ and $f_1:X_1\to X_2$ is a homotopy equivalence. 
Let $f^{-1}:X_2\to X_1$ be a homotopy inverse of $f_1$. 
We have
\begin{align*}
&\{f_n,\dots,f_3,f_2\circ\Sigma^{m_2}f_1,f_0\}^{(\ddot{s}_t)}_{(m_n,\dots,m_2,m_0)}\\
&=\{f_n,\dots,f_3,f_2\circ\Sigma^{m_2}f_1,f_1^{-1}\circ f_1\circ f_0\}^{(\ddot{s}_t)}_{(m_n,\dots,m_2,m_0)}\quad (\text{by \cite[Lemma 6.2]{OO3}})\\
&\subset 
\{f_n,\dots,f_3,f_2\circ\Sigma^{m_2}f_1\circ\Sigma^{m_2}f^{-1}_1,f_1\circ f_0\}^{(\ddot{s}_t)}_{(m_n,\dots,m_2,m_0)}\quad(\text{by (1)})\\
&=\{f_n,\dots,f_3,f_2,f_1\circ f_0\}^{(\ddot{s}_t)}_{(m_n,\dots,m_2,m_0)}\quad (\text{by \cite[Lemma 6.2]{OO3}})\\
&\subset \{f_n,\dots,f_3,f_2\circ\Sigma^{m_2}f_1,f_0\}^{(\ddot{s}_t)}_{(m_n,\dots,m_2,m_0)}
\quad(\text{by (1)}).
\end{align*}
Hence we obtain the desired equality in (2). 
\end{proof}

\section{Diagrams}
\subsection{\fbox{Diagram $D_{r+1}$}}
The diagram is commutative except for two squares marked with $(\#)$. 
The exceptional two squares are homotopy commutative. 
$$
\xymatrix{
& &&&\Sigma^{m_{r+1}}C_{r+1,r+1}\ar[d]^-{\widetilde{\Sigma}^{m_{r+1}}a_{r+1,r}} \ar@{-}[r]^-{i_{\Sigma^{m_{r+1}}j_{r+1,r}}} 
&\\
\Sigma^{r-1}\Sigma^{m_{[r+1,1]}}X_1 
\ar[rr]^-{\widetilde{\Sigma}^{m_{r+1}}g_{r+1,r}}
\ar[d]_-{\tau_{r+1}} &&\Sigma^{m_{r+1}}C_{r+1,r} 
\ar[rr]_-{i_{\widetilde{\Sigma}^{m_{r+1}}g_{r+1,r}}} \ar[urr]^-{\Sigma^{m_{r+1}}j_{r+1,r}}\ar@{=}[d] &&C_{\widetilde{\Sigma}^{m_{r+1}}g_{r+1,r}} \ar[d]^-{1\cup C\tau_{r+1}}
\ar@{-}[r]^-{i_i}
&\\
\ar@{}[rrd]|{(\#)}\Sigma^{m_{r+1}}\Sigma^{r-1}\Sigma^{m_{[r,1]}}X_1\ar[rr]^-{\Sigma^{m_{r+1}}g_{r+1,r}} \ar[d]_-{\tau'_{r+1}} &&\Sigma^{m_{r+1}}C_{r+1,r}\ar[rr]^-{i_{\Sigma^{m_{r+1}}g_{r+1,r}}} \ar@{=}[d]&&C_{\Sigma^{m_{r+1}}g_{r+1,r}}\ar[d]^-{\Phi(\widetilde{\Sigma}^{m_{r+1}}H^{r+1})}
\ar@{-}[r]^-{i_i}
&\\
\Sigma^{m^*_r}\Sigma^{r-2}\Sigma^{m^*_{[r-1,1]}}X^*_1\ar[rr]^-{\Sigma^{m^*_r}g^*_{r,r-1}} \ar[d]_-{\tau''_{r+1}} &&\Sigma^{m^*_r}C^*_{r,r-1}\ar[rr]^-{i_{\Sigma^{m^*_r}g^*_{r,r-1}}} \ar@{=}[d] &&C_{\Sigma^{m^*_r}g^*_{r,r-1}} \ar[d]^-{1\cup C\tau''_{r+1}} 
\ar@{-}[r]^-{i_i}
&\\
\Sigma^{r-2}\Sigma^{m^*_{[r,1]}}X^*_1\ar[rr]^-{\widetilde{\Sigma}^{m^*_r}g^*_{r,r-1}} &&\Sigma^{m^*_r}C^*_{r,r-1}\ar[rr]^-{i_{\widetilde{\Sigma}^{m^*_r}g^*_{r,r-1}}} \ar[rrd]_-{\Sigma^{m^*_r}j^*_{r,r-1}} &&C_{\widetilde{\Sigma}^{m^*_r}g^*_{r,r-1}}
\ar@{-}[r]^-{i_i}
&\\
& && &\Sigma^{m^*_r}C^*_{r,r} \ar[u]_-{\widetilde{\Sigma}^{m^*_r}a^*_{r,r-1}}
\ar@{-}[r]^-{i_{\Sigma^{m^*_r}j^*_{r,r-1}}}
&
}
$$
$$
\xymatrix{
&\ar[rr]^-{i_{\Sigma^{m_{r+1}}j_{r+1,r}}}&&C_{\Sigma^{m_{r+1}}j_{r+1,r}}\ar[d]_-{\widetilde{\Sigma}^{m_{r+1}}a_{r+1,r}\cup C1_{\Sigma^{m_{r+1}}C_{r+1,r}}} \ar[drr]^-{\widetilde{\Sigma}^{m_{r+1}}\omega_{r+1,r}} &&\\
&\ar[rr]^-{i_i}&&C_{i_{\widetilde{\Sigma}^{m_{r+1}}g_{r+1,r}}} \ar[rr]^-{q'_{\widetilde{\Sigma}^{m_{r+1}}g_{r+1,r}}} \ar[d]_-{(1\cup C\tau_{r+1})\cup C1_{\Sigma^{m_{r+1}}C_{r+1,r}}} &&\Sigma\Sigma^{r-1}\Sigma^{m_{[r+1,1]}}X_1\ar[d]^-{\Sigma\tau_{r+1}}\\
&\ar[rr]^-{i_i}&&\ar@{}[rrd]|{(\#)} C_{i_{\Sigma^{m_{r+1}}g_{r+1,r}}} \ar[rr]^-{q'_{\Sigma^{m_{r+1}}g_{r+1,r}}} \ar[d]_-{\Phi'(\widetilde{\Sigma}^{m_{r+1}}H^{r+1})}&&\Sigma\Sigma^{m_{r+1}}\Sigma^{r-1}\Sigma^{m_{[r,1]}}X_1\ar[d]^-{\Sigma\tau'_{r+1}}\\
&\ar[rr]^-{i_i}&&C_{i_{\Sigma^{m^*_r}g^*_{r,r-1}}}\ar[rr]^-{q'_{\Sigma^{m^*_r}g^*_{r,r-1}}} \ar[d]_-{(1\cup C\tau''_{r+1})\cup C1}&& \Sigma\Sigma^{m^*_r}\Sigma^{r-2}\Sigma^{m^*_{[r-1,1]}} X^*_1\ar[d]^-{\Sigma\tau''_{r+1}}\\
&\ar[rr]^-{i_i}&&C_{i_{\widetilde{\Sigma}^{m^*_r}g^*_{r,r-1}}} 
\ar[rr]^-{q'_{\widetilde{\Sigma}^{m^*_r}g^*_{r,r-1}}}&&\Sigma\Sigma^{r-2}\Sigma^{m^*_{[r,1]}}X^*_1\\
&\ar[rr]^-{i_{\Sigma^{m^*_r}j^*_{r,r-1}}}&&C_{\widetilde{\Sigma}^{m^*_r}j^*_{r,r-1}}\ar[u]^-{\widetilde{\Sigma}^{m^*_r}a^*_{r,r-1}\cup C1_{\Sigma^{m^*_r}C^*_{r,r-1}}} \ar[urr]_-{\widetilde{\Sigma}^{m^*_r}\omega^*_{r,r-1}} &&
}
$$
\subsection{\fbox{Diagram $D^*_5$}}
The diagram is commutative except for two squares marked with $(\#)$. 
The exceptional two squares are homotopy commutative. 
$$
\xymatrix{
& &\Sigma^{m^*_4}C^*_{4,4}\ar[d]^-{\widetilde{\Sigma}^{m^*_4}a^*_{4,3}} 
\ar@{-}[r]^-{i_{\Sigma^{m^*_4}j^*_{4,3}}} 
&\\
\Sigma^2\Sigma^{m^*_{[4,1]}}X^*_1\ar[r]^-{\widetilde{\Sigma}^{m^*_4}g^*_{4,3}} 
\ar[d]_-{\tau_5} &\Sigma^{m^*_4}C^*_{4,3}\ar[r]_-{i_{\widetilde{\Sigma}^{m_4^*}g^*_{4,3}}} \ar[ur]^-{\Sigma^{m^*_4}j^*_{4,3}}\ar@{=}[d] &\Sigma^{m^*_4}C^*_{4,3}\cup_{\widetilde{\Sigma}^{m^*_4}g^*_{4,3}}C\Sigma^2\Sigma^{m^*_{[4,1]}}X^*_1 
\ar[d]^-{1\cup C\tau_5}
\ar@{-}[r]^-{i_i}
&\\
\ar@{}[rd]|{(\#)} \Sigma^{m^*_4}\Sigma^2\Sigma^{m^*_3}X^*_1\ar[r]^-{\Sigma^{m^*_4}g^*_{4,3}} \ar[d]_-{\tau'_5} &\Sigma^{m^*_4}C^*_{4,3}\ar[r]^-i \ar@{=}[d]&\Sigma^{m^*_4}C^*_{4,3}\cup_{\Sigma^{m^*_4}g^*_{4,3}}C\Sigma^{m^*_4}\Sigma^2\Sigma^{m^*_3}X^*_1\ar[d]^-{\Phi(\widetilde{\Sigma}^{m_5}H^5)}
\ar@{-}[r]^-{i_i}
&\\
\Sigma^{m_5}\Sigma^3\Sigma^{m_{[4,1]}}X_1\ar[r]^-{\Sigma^{m_5}g_{5,4}} \ar[d]_-{\tau''_5} &\Sigma^{m_5}C_{5,4}\ar[r]^-i \ar@{=}[d] &\Sigma^{m_5}C_{5,4}\cup_{\Sigma^{m_5}g_{5,4}}C\Sigma^{m_5}\Sigma^3\Sigma^{m_{[4,1]}}X_1\ar[d]^-{1\cup C\tau''_5} 
\ar@{-}[r]^-{i_i}
&\\
\Sigma^3\Sigma^{m_{[5,1]}}X_1\ar[r]^-{\widetilde{\Sigma}^{m_5}g_{5,4}} &\Sigma^{m_5}C_{5,4}\ar[r]^-i \ar[rd]_-{\Sigma^{m_5}j_{5,4}} &\Sigma^{m_5}C_{5,4}\cup_{\widetilde{\Sigma}^{m_5}g_{5,4}}C\Sigma^3\Sigma^{m_{[5,1]}}X_1
\ar@{-}[r]^-{i_i}
&\\
& & \Sigma^{m_5}C_{5,5} \ar[u]_-{\widetilde{\Sigma}^{m_5}a_{5,4}}
\ar@{-}[r]^-{i_{\Sigma^{m_5}j_{5,4}}}
&
}
$$
$$
\xymatrix{
&\ar[r]^-{i_{\Sigma^{m^*_4}j^*_{4,3}}}&\Sigma^{m^*_4}C^*_{4,4}\cup C\Sigma^{m_4^*}C^*_{4,3}\ar[d]_-{\widetilde{\Sigma}^{m^*_4}a^*_{4,3}\cup C1_{\Sigma^{m^*_4}a^*_{4,3}}} \ar[dr]^-{\widetilde{\Sigma}^{m^*_4}\omega^*_{4,3}} &\\
&\ar[r]^-{i_i}&(\Sigma^{m^*_4}C^*_{4,3}\cup_{}C\Sigma^2\Sigma^{m^*_{[4,1]}}X^*_1)\cup C\Sigma^{m^*_4}C^*_{4,3} \ar[r]_-{q'_{\widetilde{\Sigma}^{m_4^*}g^*_{4,3}}} \ar[d]_-{(1\cup C\tau)\cup C1_{\Sigma^{m^*_4}C^*_{4,3}}} &\Sigma^3\Sigma^{m_{[5,4]}}\Sigma\Sigma^{m_{[3,1]}}X_1\ar[d]^-{\Sigma\tau_5}\\
&\ar[r]^-{i_i}&\ar@{}[rd]|{(\#)}(\Sigma^{m^*_4}C^*_{4,3}\cup C\Sigma^{m^*_4}\Sigma^2\Sigma^{m^*_3}X^*_1)\cup C\Sigma^{m^*_4}C^*_{4,3}\ar[r]^-{q'_{\Sigma^{m^*_4}g^*_{4,3}}} \ar[d]_-{\Phi'(\widetilde{\Sigma}^{m_5}H^5)}&\Sigma\Sigma^{m^*_4}\Sigma^2\Sigma^{m^*_3}X^*_1\ar[d]^-{\Sigma\tau'_5}\\
&\ar[r]^-{i_i}&(\Sigma^{m_5}C_{5,4}\cup_{}C\Sigma^{m_5}\Sigma^3\Sigma^{m_{[4,1]}}X_1)\cup C\Sigma^{m^*_4}C^*_{4,3}\ar[r]^-{q'_{\Sigma^{m_5}g_{5,4}}} \ar[d]_-{(1\cup C\tau''_5)\cup C1}& \Sigma\Sigma^{m_5}\Sigma^3\Sigma^{m_{[4,1]}}X_1\ar[d]^-{\Sigma\tau''_5}\\
&\ar[r]^-{i_i}&(\Sigma^{m_5}C_{5,4}\cup_{}C\Sigma^3\Sigma^{m_{[5,1]}}X_1)\cup C\Sigma^{m_5}C_{5,4}
\ar[r]^-{q'_{\widetilde{\Sigma}^{m_5}g_{5,4}}}&\Sigma^4\Sigma^{m_{[5,1]}}X_1\\
&\ar[r]^-{i_{\Sigma^{m_5}j_{5,4}}}&\Sigma^{m_5}C_{5,4}\cup C\Sigma^{m_5}C_{5,4}\ar[u]^-{\widetilde{\Sigma}^{m_5}a_{5,4}\cup C1_{\Sigma^{m_5}C_{5,4}}} \ar[ur]_-{\widetilde{\Sigma}^{m_5}\omega_{5,4}} &
}
$$

\section{Corrections of errors in \cite{OO1,OO2}}
\subsection{Typos in \cite{OO1}}
We should remove $\eta_5$ from Proposition 7.5 which are on pages 14 and 70. 
``Proposition 8.1'' in Remark 8.3 on page 72 should be ``Proposition 8.2''.

\subsection{Corrections of errors in \cite{OO2}}
The first paragraph of Section 4 of \cite{OO2} should be changed as follows: 
In this section we will work in $\mathrm{TOP}^w$. 
Hence $i_f:Y\to Y\cup_f CX$ is always a free cofibration for every map $f:X\to Y$ by Corollary 2.3(2),(3). 

In this connection, we must make four alterations (a)-(d) for \cite{OO2}: (a) Two words ``a cofibration" in Definition 4.1 and Lemma 4.2 should be replaced by ``a free cofibration". 
(b) We should remove ``If $j$ is a free cofibration, then" from Lemma 4.3(5). 
(c) In the proof of Lemma 4.3(1), we should change ``Since $j$ is a cofibration by the assumption, there exists a homotopy" into ``Since $j$ is a pointed cofibration, there exists a pointed homotopy". 
(d) In the proof of Lemma 4.3(5), we should change ``Suppose 
that $j$ is a free cofibration. Then $\Sigma^\ell j$'' into ``Since $j$ is a free cofibration, $\Sigma^\ell j$''. 
\end{appendix}

\end{document}